\documentclass[11pt, a4paper]{article}
\usepackage{latexsym}
\usepackage{amsmath,amsfonts,amssymb,mathrsfs,graphicx}
\usepackage{amsthm}
\usepackage{xcolor}
\usepackage{tikz}

\usepackage{bm}

\newcommand{\E}{\mathbb{E}}
\renewcommand{\P}{\mathbb{P}}
\renewcommand{\epsilon}{\varepsilon}

\DeclareMathOperator{\eps}{\varepsilon}
\DeclareMathOperator{\diam}{diam}

\DeclareMathOperator{\R}{\mathbb{R}}
\DeclareMathOperator{\N}{\mathbb{N}}
\DeclareMathOperator{\bb}{\mathbf{b}}

\newcommand{\ba}{\mathbf{a}}
\newcommand{\bi}{\mathbf{i}}
\newcommand{\bj}{\mathbf{j}}

\newcommand{\bl}{\mathbf{l}}
\newcommand{\e}{\epsilon}

\newcommand\ubd{\overline{\mbox{\rm dim}}_{\rm B}\,} % upper box dimension
\newcommand\lbd{\underline{\mbox{\rm dim}}_{\rm B}\,} % lower box dimension
\newcommand\bdd{\mbox{\rm dim}_{\rm B}}% box dimension
\newcommand\pkd{\mbox{\rm dim}_{\rm P}\,} % packing dimension
\newcommand\hdd{\mbox{\rm dim}_{\rm H}\,} % Hausdorff dimension

\newcommand{\be}{\begin{equation}} % begin equation
\newcommand{\ee}{\end{equation}} % end equation
\renewcommand{\i}{{\bf i}} % bold i

\newtheorem{theo}{Theorem}[section]
\newtheorem{cor}[theo]{Corollary}
\newtheorem{lem}[theo]{Lemma}
\newtheorem{prop}[theo]{Proposition}

\theoremstyle{definition}

\title{Box-counting dimension in one-dimensional random geometry of multiplicative cascades}
\author{Kenneth J. Falconer$^1$ and Sascha Troscheit$^{2,}$\thanks{ST was funded by Austrian Research Fund
  (FWF) Grant M-2813.}\\\\
  \textit{$^1$Mathematical Institute,}\\ 
  \textit{University of St Andrews,}\\
  \textit{St Andrews, KY16 9SS, Scotland.}\\
  \texttt{kjf@st-andrews.ac.uk}\\\\
  \textit{$^2$Faculty of Mathematics,}\\
  \textit{University of Vienna,}\\
  \textit{Oskar Morgenstern Platz 1, 1090 Wien,
  Austria.}\\
\texttt{sascha@troscheit.eu}}

\date{\today}

\begin{document}
\maketitle
\allowdisplaybreaks

\begin{abstract}
  We investigate the box-counting dimension of the image of a set $E \subset \mathbb{R}$ under a
  random multiplicative cascade function $f$. The corresponding result for Hausdorff dimension was
  established  by Benjamini and Schramm in the context of random geometry, and for sufficiently
  regular sets, the same formula holds for the box-counting dimension. However, we show that this is far
  from true in general, and we compute explicitly a formula of a very different nature that gives the
  almost sure box-counting dimension of the random image $f(E)$ when the set $E$ comprises a
  convergent sequence. In particular, the box-counting dimension of $f(E)$ depends more subtly on $E$
  than just on its dimensions. We also obtain lower and upper bounds for the box-counting dimension
  of the random images for general sets $E$.
  \renewcommand{\thefootnote}{\fnsymbol{footnote}} 
  \renewcommand{\thefootnote}{\arabic{footnote}} 
\end{abstract}

\section{Introduction}
The random multiplicative cascade is a well-studied random measure on the unit cube in $d$-dimensional
Euclidean space. It originally arose in Mandelbrot's study of turbulence \cite{Mandelbrot74}
but has since been investigated in its own right, see e.g. \cite{Barral10,
Barral10a,Barral14,Barral18,Falconer16,Jin14,Kahane76,Molchan04}.
In one dimension the measure may be constructed iteratively by subdividing the unit line into dyadic intervals,
multiplying the length of each subdivision by an i.i.d.\ copy of a common positive random variable $W$ with
mean $\E(W)=1$. The resulting measure $\mu$ can alternatively be thought of in terms of its cumulative distribution function
$f(x)=\mu([0,x))$ which may also be interpreted as a random metric by setting $d(x,y) =
|f(x)-f(y)|$. The latter approach was picked up as a model for quantum gravity by Benjamini and
Schramm \cite{Benjamini09}, who analysed the change in Hausdorff dimension of deterministic subsets
$E\subset[0,1]$
under the random metric, or equivalently, its image under $f$ with the Euclidean
metric. They obtained an elegant formula for the almost sure Hausdorff dimension $s$ of $F$
with respect to the random metric in terms of the Hausdorff dimension $d$ of $F$ in the Euclidean metric
and the moments of $W$:
\begin{equation}\label{eq:Benjamini}
  2^d = \frac{2^s}{\E(W^s)}.
\end{equation}
Further, when $W$ has a log-normal distribution, they showed that
the formula reduces to the
famous KPZ equation, first established by Knizhnik, Polyakov, and Zamolodchikov \cite{Knizhnik88}, 
that links the dimensions of an object in deterministic and quantum gravity
metrics.
Barral et al.\ \cite{Barral14} removed some of the assumptions of Benjamini and Schramm, and
Duplantier and Sheffield \cite{Duplantier11} studied the same phenomenon in another popular model of quantum gravity,
Liouville quantum gravity. Duplantier and Sheffield show that a KPZ formula holds for the \textit{Euclidean
expectation dimension}, an ``averaged'' box-counting type dimension.

Using dimensions to study random geometry has a fruitful history, see e.g.
\cite{Aldous91, Benjamini09, Ding20, Gwynne20,LeGall13, Troscheit21}, which use dimension theory
in their methodology.
Whilst much of the literature in random geometry considers Hausdorff dimension or other
`regular' scaling dimensions, box-counting dimensions have not been explored as thoroughly. In
part this may be due to the more complicated geometrical properties of box-counting dimension of a set,
manifested, for instance, in its projection properties, see \cite{fal:2021}. 

One might hope that a formula analogous to \eqref{eq:Benjamini} would also hold for the box-counting
dimension of images of sets under the cascade function $f$. We investigate this question and find
that this need not be the case for sets that are not sufficiently
homogeneous. We give bounds that are valid for the box-counting dimensions of $f(E)$ for general
sets $E$, and then in Theorems \ref{thm:dimea} and \ref{thm:dimgen} give an exact formula for the
box dimension of $f(E)$ for a large family of sets of a very different form from
\eqref{eq:Benjamini}.

We remark that the study of dimensions of the images of sets under various random functions goes
back a considerable time. For example, with $B_\alpha:\R\to\R$ as index-alpha fractional Brownian
motion, $\hdd B_\alpha (E) = \min\{1,\frac{1}{\alpha}\hdd E\}$, see Kahane \cite{Kah}. On the other
hand, the corresponding result for packing and box-counting dimensions is more subtle, depending on
`dimension profiles', as demonstrated by Xiao \cite{Xi}.

\subsection{Notation and definitions}
This section introduces random multiplicative cascade functions and dimensions along with the notation that we shall use.
We will use finite and infinite words from the alphabet $\{0,1\}$ throughout.  We write finite words as
$\bi= i_1i_2\ldots i_k \in\{0,1\}^k$ for $k\in \mathbb{N}$ with $\varnothing$ as the empty word,
with  $\{0,1\}^* =\bigcup_0^\infty\{0,1\}^k$, and $\bi= i_1i_2\ldots \in\{0,1\}^\mathbb{N}$ for the
infinite words. We combine words by juxtaposition, and write $|\bi|$ for the length of a finite
word.

For $\bi= i_1i_2\ldots i_k \in\{0,1\}^k$  let $I_\bi$ denote the dyadic interval 
$$I_\bi = \Big[\sum_{j=1}^k 2^{-j} i_j, \sum_{j=1}^k 2^{-j} i_j +2^{-k}\Big),$$ 
taking the rightmost intervals $[1-2^{k},1]$ to be
closed. We denote the set of such dyadic intervals of lengths $2^{-k}$ by ${\mathcal I}_k$. Note
that every interval of  ${\mathcal I}_k$ is the union of exactly two disjoint intervals in
${\mathcal I}_{k+1}$.

Underlying the random cascade construction is a random variable $W$, with $\{W_{\bi}
: \bi \in \{0,1\}^*\}$ a tree of independent random variables with the distribution of $W$. We will
assume throughout that $W$ is positive, not almost-surely constant and that
\be\label{standassump} 
\E(W)=1\quad\text{and}\quad \E(W\log_2 W)\leq1.
\ee
Note $\E(W\log_2 W) \leq1$ implies $\E(W^t) < \infty$ for $t\in[0,1]$.

We differentiate between the subcritical regime when $\E(W\log_2W)<1$ and
the critical regime when $\E(W\log_2 W)=1$.
Unless otherwise noted, we assume the subcritical regime. 
Here, the length of the random image  $f([0,1])$ is given by 
$$ 
L := |f([0,1])| = \mu[0,1]= \lim_{k\to\infty}\sum_{\bi \in\{0,1\}^k}2^{-k}W_{ i_1}\dots W_{i_1 \dots i_k},
$$
where $|A|$ denotes the diameter of a set $A$, and with $\mu$ the (subcritical) random cascade measure.
Comprehensive accounts of the properties of $L$ can be found in \cite{Benjamini09} and
\cite{Kahane76}, in particular the assumption that $\E(W\log_2 W)<1$ implies that $L$ exists and
$0<L<\infty$ almost surely and $\E(L)=1$.
Similarly, the length of the random image of the interval $I_\bi \in {\mathcal I}_k$ is given by
$$
|f(I_\bi)| = \mu(I_\bi)= 2^{-k}W_{i_1}\dots W_{ i_1 \dots i_k} L_{\bi} \quad \text{ where } \quad 
L_{\bi} := \lim_{n\to\infty}\sum_{\bj \in\{0,1\}^n} 2^{-n}W_{\bi j_1}\dots W_{\bi j_1 \dots j_n}
$$
has the distribution of $L$, independently for $\bi \in\{0,1\}^k$ for each fixed $k$. The random
{\em multiplicative cascade measure} $\mu$ on $[0,1]$ is obtained by extension from the
$\mu(I_\bi)$. Almost surely,  $\mu$ has no atoms and $\mu(I)>0$ for every interval $I$, so the
associated random {\em multiplicative cascade function} $f:[0,1] \to \mathbb{R}^{\geq 0}$ given by
$f(x)=\mu([0,x))$ is almost surely strictly increasing and continuous. We do not need to refer to
$\mu$ further and will work entirely with $f$.

In the critical regime a similar measure exists. In particular, normalising with $\sqrt{k}$ gives
\[
  L = |f([0,1])| = \mu[0,1] = \lim_{k\to\infty}\sqrt{k}\sum_{\bi \in\{0,1\}^k}2^{-k}W_{ i_1}\dots
  W_{i_1 \dots i_k},
\]
where the convergence is in probability. The random limit $L$ exists and $0<L<\infty$ almost surely
under the additional assumption that $\E(W\log^2 W)<\infty$, see \cite{Boutaud19}.
Here $\E(L)=\infty$, unlike the subcritical case. 
The associated measure $\mu$ is therefore finite almost surely, and it was shown
in \cite{Barral14} that this measure almost surely has no atoms.
We refer the reader to \cite{Barral14} for a detailed account of critical Mandelbrot cascades.
Note further that the length of the random image of the interval $I_{\bi}$ is given by
\[
  |f(I_\bi)| = \mu(I_\bi)= \sqrt{k}\cdot2^{-k}W_{i_1}\dots W_{ i_1 \dots i_k} L_{\bi}, 
\]
where $L_{\bi}$ is a random variable that is equal to $L$ in distribution (and hence has infinite
mean).

Note that while we will consider image sets $f(E)$ as subsets of $\R$ with the Euclidean metric, equivalently one could
define a random metric $d_W$ by setting $d_W(x,y) = |f(x)-f(y)| = \mu([x,y])$ and
investigate $(E,d_W)$ instead. For more details on such alternative interpretations, see \cite{Benjamini09}.

The Hausdorff dimension $\hdd$ is the most commonly considered form of fractal dimension. The {\em
Hausdorff dimension} of a subset $E$ of a metric space $(X,d)$  may be defined as
\[
  \hdd E = \inf\Big\{ \alpha >0 : \text{ for all } \eps>0,  \text{ there is a cover }
  (U_i)_{i=1}^{\infty} \text{ of } E \text{ such that }\sum_{i=1}^\infty \diam(U_i)^\alpha <\eps \Big\}.
\]
%where the infimum is taken over all open covers of $X$.
Perhaps more intuitive are the box-counting dimensions. Let $(X,d)$ be a 
metric space and $E\subset X$ be non-empty and bounded. Write $N_r(E)$ for the minimal number of
sets of diameter at most $r>0$ needed to
cover $E$. The {\em upper} and {\em lower box-counting dimensions} (or {\em box dimensions}) are given by
\[
  \lbd E = \liminf_{r\to0}\frac{\log N_{r}(E)}{-\log r}, \qquad
  \ubd E = \limsup_{r\to0}\frac{\log N_r(E)}{-\log r}.
\]
If this limit exists, we speak of the {\em box-counting dimension} $\bdd E$ of $E$. Note that 
whilst many `regular' sets (such as Ahlfors regular sets) have equal Hausdorff and box-counting
dimension this is not true in general.

\subsection{Statement of results}

Our aim is to find or estimate the dimensions of $f(E)$ where $f$ is the random
cascade function and $E\subset [0,1]$. Note that these dimensions are tail events, since changing
$\{W_{\bi} : |\bi \leq k\}$ for a fixed $k$ results in just a bi-Lipschitz distortion of the set
$f(E)$. This implies that the Hausdorff and upper and lower box-counting dimensions of $f(E)$ each
take an almost sure value. 

Benjamini and Schramm  established the formula for the Hausdorff dimension.

\begin{theo}[Benjamini, Schramm \cite{Benjamini09}]\label{thm:Benjamini}
  Let $f$ be the distribution of a subcritical random cascade.
  Suppose that $\E(W^{-t})<\infty$ for all $t\in[0,1)$ in addition to the standard assumptions \eqref{standassump}.
  Let $E\subset [0,1]$ and write $d_E = \hdd E$. 
  Then the almost sure Hausdorff dimension $\hdd f(E)$ of the random image of $E$
  is the unique value $s$ that satisfies
  \begin{equation}\label{eq:BenjaminiFormula}
    2^{d_E} = \frac{2^{s}}{\E(W^{s})}.
  \end{equation}
\end{theo}

Note that the expression on the right in $\eqref{eq:BenjaminiFormula}$ is continuous in $s$ and
strictly increasing, mapping $[0,1]$ onto $[1,2]$, see  \cite[Lemma 3.2]{Benjamini09}.

This result was improved upon by Barral et al.\ who also proved the result for the critical cascade measure.
\begin{theo}[Barral, Kupiainen, Nikula, Saksman, Webb \cite{Barral14}]\label{thm:Barral}
  Let $f$ be the distribution of a subcritical or critical random cascade.
  Assume that $\E(W^{-t})<\infty$ for all $t\in(0,\tfrac12)$ and $\E(W^{1+\eps})<\infty$ for some $\eps>0$.
  Let $E\subset[0,1]$ be some Borel set
  with Hausdorff dimension $d_E = \hdd E$.
  Then the almost sure Hausdorff dimension $\hdd f(E)$ of the random image of $E$
  is the unique value $s$ that satisfies
  \[
    2^{d_E} = \frac{2^{s}}{\E(W^{s})}.
  \]
\end{theo}

\subsubsection{General bounds for box-counting dimensions of images}
Our first result is that the upper box-counting dimension of $E$ is bounded above by a value
analogous to that in  \eqref{eq:BenjaminiFormula}, though the assumption that
$\E(W^{-t})<\infty$ for $t>0$ is not required here for subcritical cascades.

\begin{theo}{\rm (General upper bound)}\label{thm:upper}
  Let $f$ be the distribution of a subcritical random cascade or the distribution of a critical
  random cascade with the additional
  assumption that $\E(W^{-t})<\infty$ for some $t>0$ and $E(W \log^2W)<\infty$.
  Let $E\subset[0,1]$ be non-empty and compact and let  $d_E=\ubd E$. Then almost surely $\ubd f(E)
  \leq  s$ where $s$ is the unique non-negative number satisfying
  \be\label{upbound}
  2^{d_E} = \frac{2^s}{\E(W^s)}.
  \ee 
\end{theo} 

Combining this result with Theorem \ref{thm:Barral} we get the immediate corollary for sets with
equal Hausdorff and (upper) box-counting dimension, such as Ahlfors regular sets.
\begin{cor}\label{thm:homogeneous}
  Let $f$ be the distribution of a subcritical or critical random cascade.
  Suppose additionally that $\E(W^{-t})<\infty$ for all $t\in(0,\tfrac12)$ and in the critical case assume
  also that $E(W\log^2 W)<\infty$. If $E\subset[0,1]$ is non-empty and compact,
  and $\hdd E = \ubd E=d_E$, then almost surely $\hdd f(E)=\ubd f(E)=s$
  where $s$ is given by \eqref{upbound}.
\end{cor}

We can also apply Theorem \ref{thm:upper} to the packing dimension.
\begin{cor}
  Let $f$ be the distribution of a subcritical cascade.
  If $E\subset[0,1]$ is non-empty and compact and $d_E= \pkd E$, then almost surely $\pkd f(E)\leq s$
  where $s$ satisfies
  \[
    2^{d_E} = \frac{2^s}{\E(W^s)}.
  \] 
\end{cor}

\begin{proof}
  Recall that the packing dimension of a set $E$ equals its modified upper box-counting dimension, that is 
  $\pkd(E) = {\overline \dim}_{\rm MB}(E) = \inf\{\sup_i E_i:E \subset \cup_{i=1}^\infty \ubd E_i\}$, where
  the $E_i$ may be taken to be compact. The conclusion follows by applying Theorem \ref{thm:upper} to
  countable coverings of $E$.
\end{proof}

We also derive general lower bounds. 
\begin{theo}{\rm (General lower bound)}\label{thm:lower}
  Let $f$ be the distribution of a subcritical random cascade. 
  Let $E\subset[0,1]$ be non-empty and compact. Then almost surely
  \be\label{lowerbound}
  \ubd f(E) \geq \frac{\ubd E}{1- \E(\log_2 W)},
  \ee 
  and, provided that additionally $\E(W^p)<\infty$ for some $p>2$, then
  \be\label{lowerbound2}
  \lbd f(E) \geq \frac{\lbd E}{1- \E(\log_2 W)}.
  \ee 
  Further, the same inequalities hold for critical random cascades under the additional assumptions
  that $\E(W^{-t})<\infty$ for some $t>0$ and $E(W\log^2 W)<\infty$.
\end{theo}

It should be noted that these upper and lower bounds are asymptotically equivalent for small dimensions.
\begin{prop}\label{thm:asymptotics}
  Let $d\in(0,1)$ and let $s_1$ be the unique solution to
  \begin{align}
    2^d = \frac{2^{s_1}}{\E(W^{s_1})} \quad&\Longleftrightarrow\quad d = s_1-\log_2\E(W^{s_1}).\label{eq:upper}
    \intertext{Further, let }
    s_2 = \frac{d}{1-\E(\log_2 W)} \quad&\Longleftrightarrow\quad d = s_2- \E(\log_2 W^{s_2}) .\label{eq:lower}
  \end{align}
  Then $s_1/s_2 \to 1$ as $d\to 0$.
\end{prop}
Theorems  \ref{thm:upper} and  \ref{thm:lower}, as well as Proposition \ref{thm:asymptotics}
will be proved in Section \ref{genbounds}.

\subsubsection{Decreasing sequences with decreasing gaps}
\label{sec:decreasingsets}
To show that neither the expressions in \eqref{upbound} nor \eqref{lowerbound}-\eqref{lowerbound2}
give the actual box dimensions of $f(E)$ for many sets $E$, and that the box dimension of the random
image $f(E)$ depends more subtly on $E$ than just on its dimension, 
we will consider sets formed by decreasing sequences that accumulate at $0$, and  obtain the almost
sure box dimensions of their images in our main Theorems \ref{thm:dimea} and \ref{thm:dimgen}.
Let $\ba = (a_n)_{n\in\N}$ be a sequence of positive reals that converge to $0$. We write $E_{\ba} = \{a_n :
n\in\N\}\cup\{0\}$. 

Given two sequences $\ba$ and $\bb$ of positive reals that are eventually decreasing and
convergent to $0$ we say that $\bb$ \emph{eventually separates} $\ba$ if there is some $n_0\in\N$
such that for all $n\geq n_0$ there exists $m\in\N$ such that $a_{n+1}\leq b_m \leq a_{n}$. We will
need this property, which is preserved under strictly increasing functions, when comparing
dimensions of the images of sequences under the random function $f$. However, we first use it to
compare the box-counting dimensions of deterministic sets. The simple proofs of the following two
lemmas are given in Section \ref{sec:otherproofs}.

\begin{lem}\label{thm:eventuallySeparates}
  Let $\ba=(a_n)_{n}$ and $\bb=(b_n)_{n}$ be strictly decreasing sequences convergent to
  $0$ such that $\bb$ eventually separates $\ba$. Then
  \[
    \lbd E_{\ba}\leq \lbd E_{\bb}\quad\text{and}\quad \ubd E_{\ba} \leq \ubd E_{\bb}.
  \]
\end{lem}

We write $S_p\ (p>0)$ for the set of sequences $ \ba\ =\ (a_n)_{n}$ convergent to 0 such that
$\frac{-\log a_n}{\log n} \to p$. We say that  the sequence $ \ba\ =\ (a_n)_{n}$ is {\em decreasing
with decreasing gaps} if $a_n \searrow 0$ and $a_n - a_{n+1}$ is (not necessarily strictly)
decreasing.
%\be\label{simdef} 
%%\ee

\begin{lem}\label{thm:decreasingseqs}
  Let $\ba=(a_n)_{n}\in S_p$ and $\bb=(b_m)_{m} \in S_q$ be decreasing sequences with decreasing gaps with $0<q<p$. Then $\bb$ eventually separates $\ba$.
\end{lem}

Of course, the most basic example of such sequences are the powers of reciprocals.
For $p>0$ let $\ba(p) = (n^{-1/p})_n \in S_p$ and let 
\[
  E_{\ba(p)} ={\textstyle \big\{0,1,\frac{1}{2^p}, \frac{1}{3^p},\ldots\big\}}\cup\{0\}.
\]
We may compare $\ba(p)$ with other sequences in $S_p$.

\begin{cor}\label{sequencedims}
  Let $\ba=(a_n)_{n} \in S_p$  be a strictly decreasing sequence with decreasing gaps such that
  $ (a_n)_{n}\in S_p$, where $p>0$.  Then 
  $$\bdd E_\ba = \frac{1}{p+1}.$$
\end{cor}

\begin{proof}[Proof of Corollary \ref{sequencedims}]
  Clearly $\ba(q) \in S_q$ for $q>0$ and
  it is well-known that $\bdd E_{\ba(q)} =1/(1+q)$, see \cite[Example 2.7]{Fa}. If $q_1<p<q_2$ then 
  $\ba(q_1)$ eventually separates $\ba$ and $\ba$ eventually separates $\ba(q_2)$, by Lemma \ref{thm:decreasingseqs}, so by Lemma \ref{thm:eventuallySeparates}, 
  $$\frac{1}{1+q_2} =\lbd (E_{\ba(q_2)}) \leq   \lbd (E_{\ba}) \leq \lbd (E_{\ba(q_1)}) =\frac{1}{1+q_1},$$
  with similar inequalities for upper box dimension. Since we may take $q_1$ and $q_2$ arbitrarily close to $p$,  the conclusion follows.
\end{proof}

\subsubsection{Random images of decreasing sequences with decreasing gaps}

We aim to find the almost sure dimension of $f(E_\ba)$ for sequences $E_\ba \in S_p$ $(p>0)$. To
achieve this we work with special sequences $E^{\alpha} \in S_{1/\alpha}$ for which $\bdd
f(E^{\alpha})$ is more tractable, and then extend these conclusions across the $S_p$ using the
eventual separation property.

%We now introduce special sequences that are central to our analysis of random images of sequences  $f(E_\ba)$.
Let $\alpha >0$ be a real parameter and let $E^{\alpha}\subset [0,1]$ be the set given in terms of binary expansions by
\[
  E^{\alpha} = \big\{0.0^{k-1}1\bj 000\cdots, \text{ for all } k\in \mathbb{N},
  \,\bj\in\{0,1\}^{\lfloor \alpha k\rfloor}\big\} \cup \{0\},
\]
where $0^{m}$ denotes $m$ consecutive 0s and $\{0,1\}^m$ represents all digit sets of length $m$ of 0s and 1s.
Equivalently, letting $ \Sigma^\alpha$ be the set of infinite strings 
\[
  \Sigma^\alpha = \big\{ 0^{k-1}1\bj00\dots\in\{0,1\}^{\mathbb{N}}\text{, for all }k\in
    \mathbb{N},\,\bj\in\{0,1\}^{\lfloor \alpha
  k\rfloor}\big\} \cup \{000\dots\},
\]
then  $E^{\alpha}$ is the  image of  $\Sigma^\alpha$ under the natural bijection $\pi(\bi) =
\sum_{n=1}^\infty i_n/2^{n}$ where $\bi = i_1i_2\ldots$, and we will identify such strings with
binary numbers in the obvious way throughout.
Clearly, $E_{\alpha}$ consists of a decreasing sequence of numbers with decreasing gaps, together with 0.  

If the $n$th term in this sequence is $\alpha_n = 0.0^{k-1}1\bj 00\cdots \in E^{\alpha}$ with
$\bj\in\{0,1\}^{\lfloor \alpha k\rfloor}$, then $2^{-(k+1)} <  \alpha_n \leq 2^{-k}$. Moreover,
$$2^{\lfloor (k-1)\alpha \rfloor}\leq 2^{\lfloor \alpha \rfloor}+ \cdots + 2^{\lfloor (k-1)\alpha \rfloor} 
\leq n \leq 2^{\lfloor \alpha \rfloor}+ \cdots + 2^{\lfloor k\alpha \rfloor} \leq 2^{(k+1)\alpha}.$$
Hence
\be\label{aln}
\frac{k}{(k+1)\alpha}   \leq  \frac{-\log_2 \alpha_n}{\log_2 n} < \frac{k+1}{\lfloor (k-1)\alpha \rfloor}.
\ee
Letting $n\to \infty$ and thus $k\to \infty$, it follows  that $(\alpha_n)_{n} \in S_{1/\alpha}$, so
by Corollary \ref{sequencedims}  $\bdd E^{\alpha} = \alpha /(1+\alpha)$.

\begin{figure}[t]
  \begin{center}
    \begin{tikzpicture}
      \node (image) at (0,0) {\includegraphics[width=.9\textwidth]{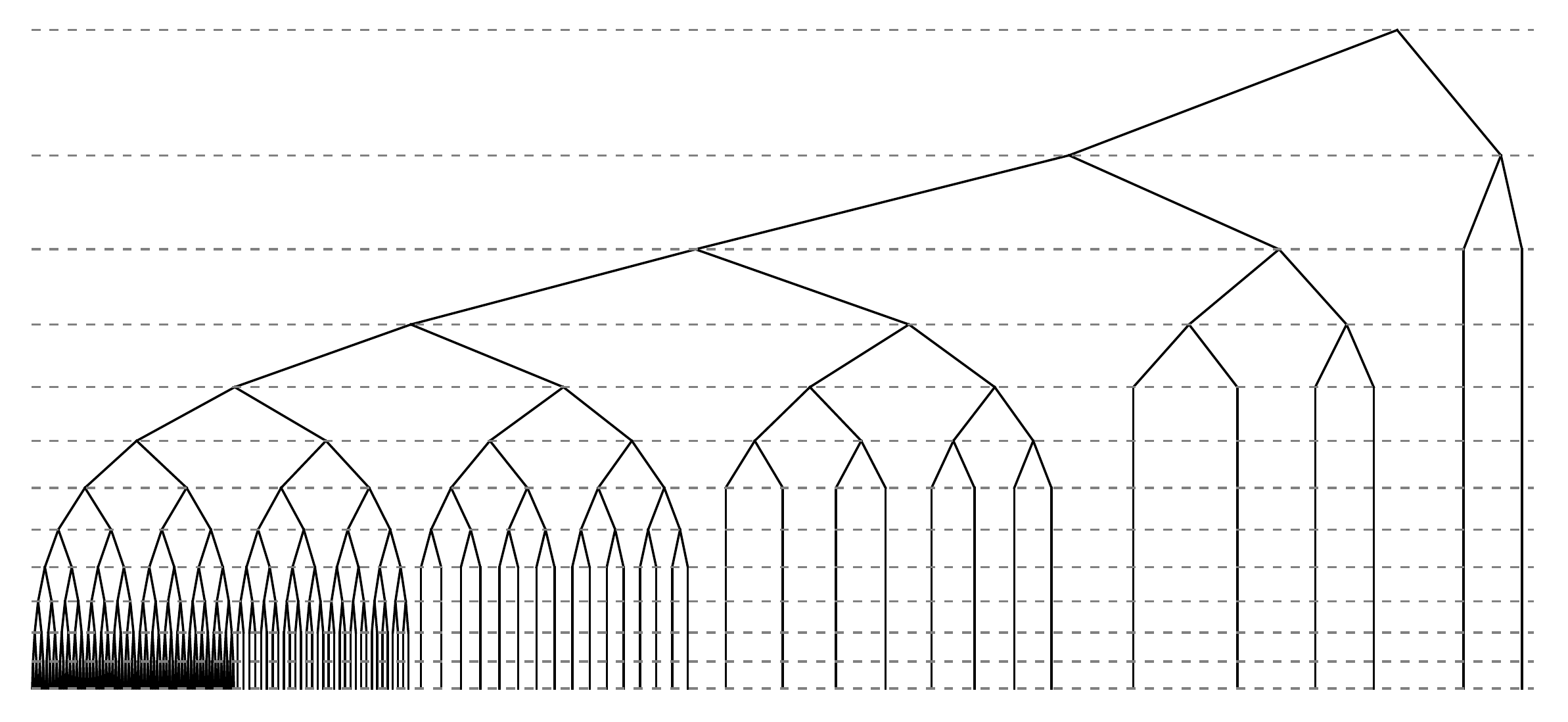}};
      \node () at (7.5,3.3) {$k=1$};
      \node () at (7.5,2.1) {$k=2$};
      \node () at (7.5,1.2) {$k=3$};
      \node () at (7.5,0.3) {$\vdots$};
      \node () at (0,-3.3) {$\vdots$};
      \node () at (-5,-3.3) {$\vdots$};
      \node () at (5,-3.3) {$\vdots$};
    \end{tikzpicture}
  \end{center}
  \caption{The coding tree of $E^\alpha$ for $\alpha = 1$. At every left-most level $k$ node a full
  binary tree of height $k$ branches off.}
  \label{fig:thyrse}
\end{figure}
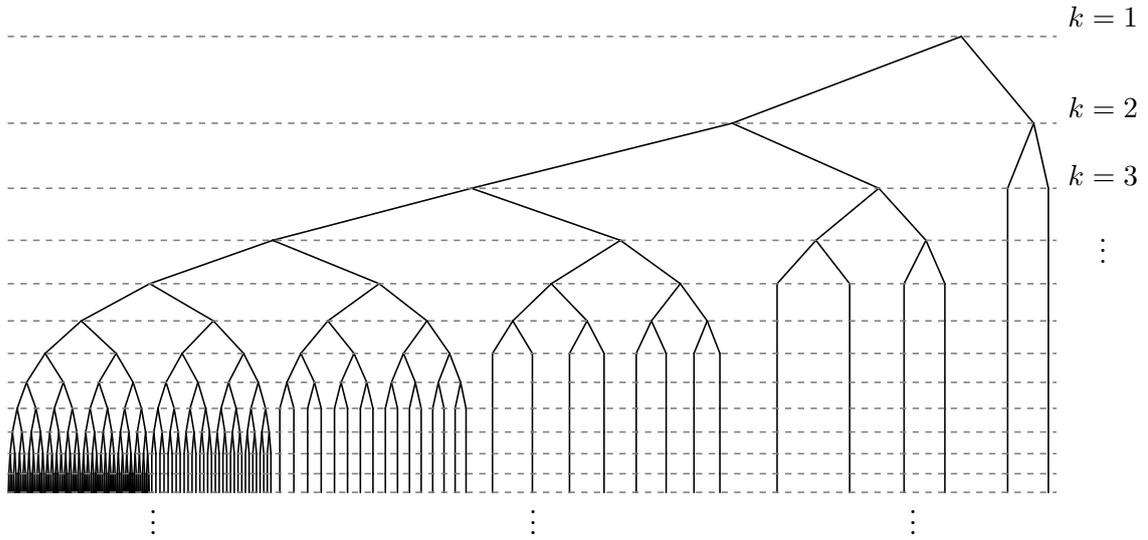

We may think of the structure of a set $E\subset [0,1]$ as a tree formed by the hierarchy of binary
intervals that overlap $E$. The structure of $E^{\alpha}$, with a `stem' at 0 and a sequence of
full trees branching off this stem, see Figure~\ref{fig:thyrse},
makes it convenient for analysing the box dimension of the random
image $f(E^{\alpha})$.
To obtain the lower bound, we will require a result on large deviations in binary trees that requires
the additional assumptions that 
\be\label{extraassump}
\E(W^t)<\infty \text{ for all }t > 0\quad\text{and}\quad \E(W^{-u})<\infty\text{ for
some } u>0.
\ee
The first condition implies 
that  $\E(W^t \log^n W )<\infty$ for all $t>0$, and in particular that $\E(W^t)$ is smooth for all $t>0$.
Applying the dominated convergence theorem, we can compute the
derivatives of the $t$-moments of $W$:
\begin{equation}\nonumber
  \tfrac{\partial}{\partial t}\E(W^t) = \E\left( \tfrac{\partial}{\partial t}W^t \right)
  = \E(W^t \log W)
  \quad \text{and} \quad
  \tfrac{\partial^2}{\partial t^2}\E(W^t) = \E(W^t\log^2 W) > 0.
\end{equation}
We also note that 
\be\label{derivratio}
\frac{\partial}{\partial t}\left(\frac{\E(W^t\log W)}{\E(W^t)}\right)
=\frac{\E(W^t \log^2 W)\E(W^t) - \E(W^t\log W)^2}{\E(W^t)^2} > 0,
\ee
so in particular $\E(W^t\log W)/\E(W^t)$ is strictly increasing in $t\geq0$, since, by the Cauchy-Schwarz inequality,
$  \E(W^t\log(W))^2 = \E(W^{t/2} W^{t/2}\log W)^2 
< \E(W^t)\E(W^t \log^2 W).$

We can now state our main results.

\begin{theo}\label{thm:dimea}
  Let $W$ be a positive random variable that is not almost surely constant and satisfies
  \eqref{standassump} and \eqref{extraassump}.
  Let $f$ be the random homeomorphism given by the (subcritical) multiplicative cascade with
  random variable $W$.
  Then, almost surely, the random image $f(E^\alpha)$ has box-counting dimension
  \be\label{dimeal}
  \bdd f(E^\alpha) =\sup_{x>0}\frac{1+\inf_{t>0}\big(x t+\log_2\E(W^t)\big)}{1+x+(1-\E(\log_2 W))/\alpha}
  \ee
  for all $\alpha>0$. We note that we only require \eqref{extraassump} for the lower bound in \eqref{dimeal}.
\end{theo}
The dimension formula is expressed in terms of the Legendre transform of the logarithmic moment
$\log_2 \E(W^t)$. Figure \ref{fig:moment} shows the logarithmic moment and its Legendre transform
for a log-normally distributed $W$ that satisfies our assumptions.
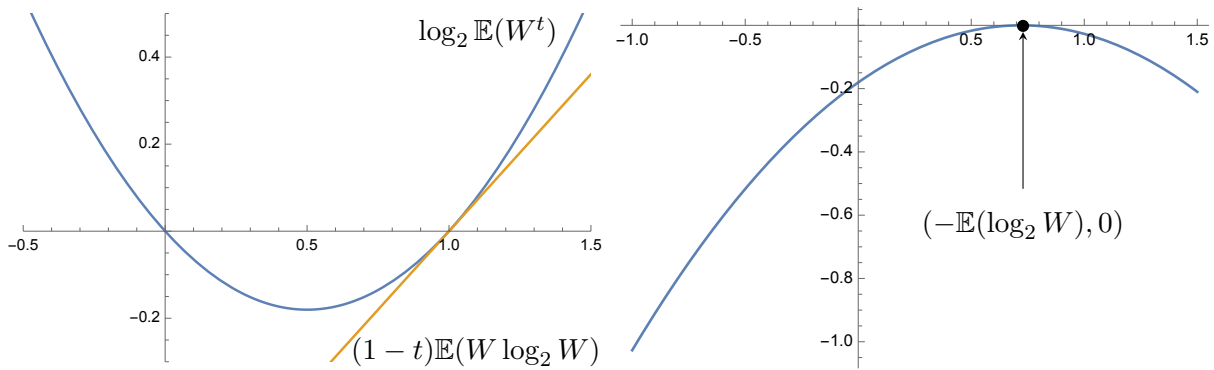
\begin{figure}[tbp]
  \begin{tikzpicture}[x=\textwidth/10, y=\textwidth/10]
    \node  (moments) at (0,0) {\includegraphics[width = 0.49 \textwidth]{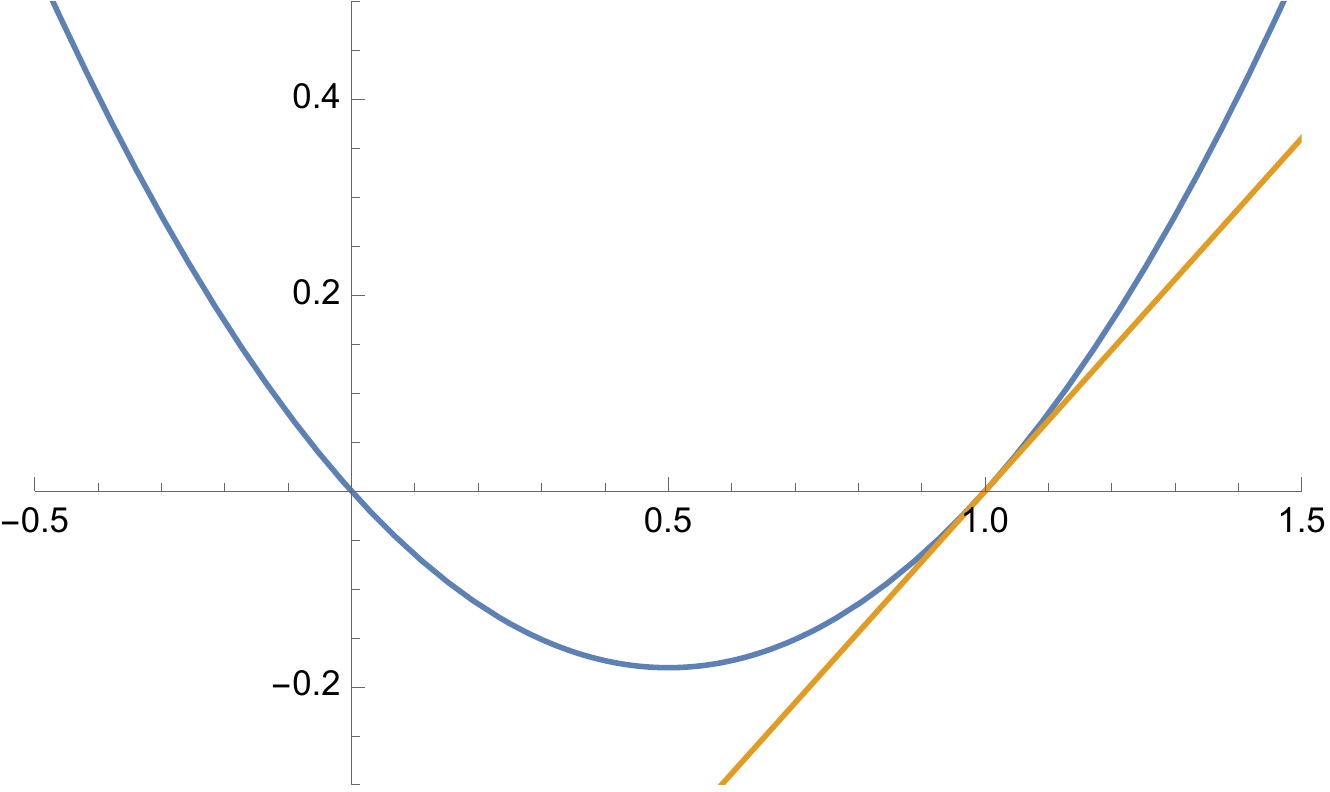}};
    \node  (legendre) at (5,0) {\includegraphics[width = 0.49 \textwidth]{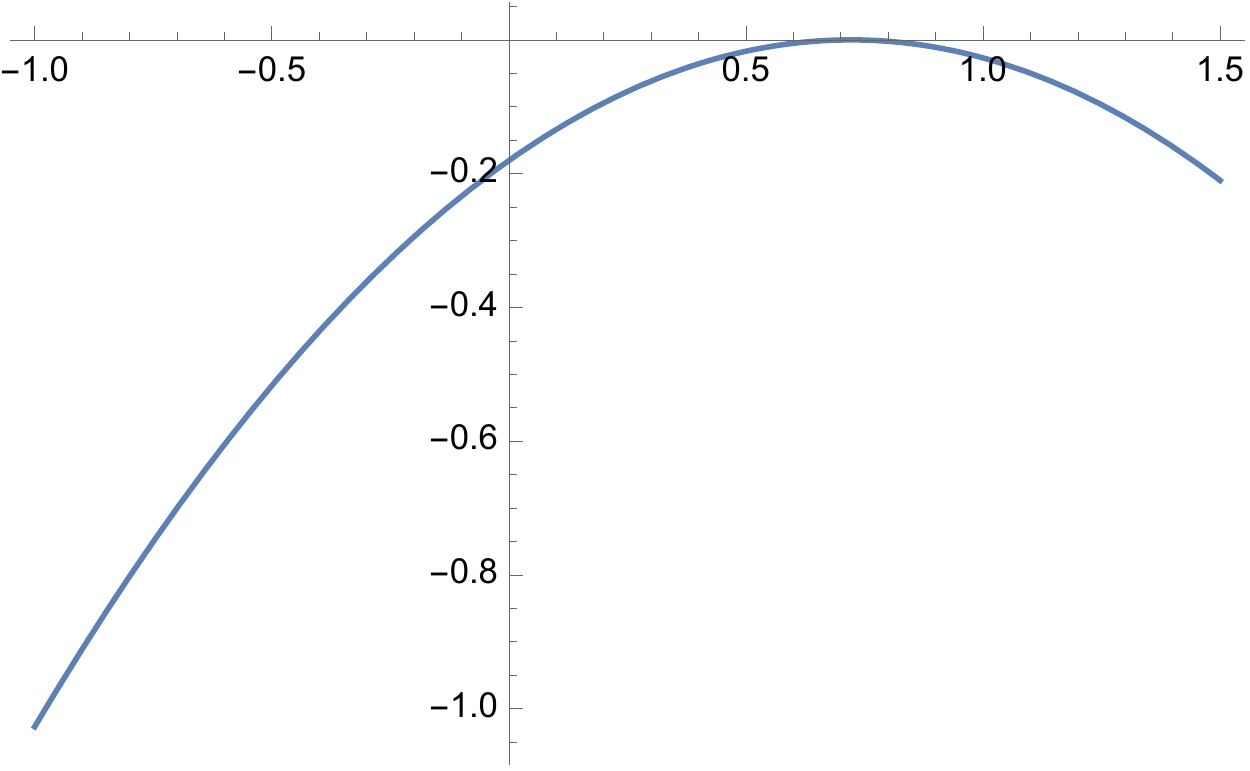}};
    \node () at (1.5,1.3) {$\log_2 \E(W^t)$};
    \node () at (1.4,-1.35) {$(1-t)\E(W\log_2 W)$};
    \fill (5.90,1.35) circle (0.05);
    \node () at (5.9,-.30) {$(-\E(\log_2 W),0)$};
    \draw [-stealth] (5.9,.0) -- (5.9,1.30);

  \end{tikzpicture}
  \caption{A plot of the moments $\log_2\E(W^t)$ (left) along with its Legendre transform (right)
  for $W$ having log-normal distribution with variation $\sigma^2 = 1$.}
  \label{fig:moment}
\end{figure}

The right hand side of \eqref{dimeal} is strictly increasing and
continuous in $\alpha$, as we verify in Lemma
\ref{thm:monotone}. Using this, and noting that the `eventually
separated' condition is preserved under monotonic increasing functions, we may compare $f(E^{1/p})$
with $f(E_{\ba})$, where $\ba \in S_p$,  to transfer this conclusion to more general
sequences. 

\begin{theo}\label{thm:dimgen}
  Let $W$ be a positive random variable that is not almost surely constant and satisfies
  \eqref{standassump} and \eqref{extraassump}.
  Let $f$ be the random homeomorphism given by the (subcritical) multiplicative cascade with
  random variable $W$. 
  Then, almost surely, the random images $f(E_{\ba})$ have box-counting dimension
  \be\label{dimgen}
  \bdd f(E_{\ba}) =\sup_{x>0}\frac{1+\inf_{t>0}\big(x t+\log_2\E(W^t)\big)}{1+x+(1-\E(\log_2 W))p}, 
  \ee
  for all decreasing sequences with decreasing gaps $\ba=(a_n) \in S_p$ and $p>0$ simultaneously.
\end{theo}

The formula in \eqref{dimgen} clearly does not coincide with \eqref{eq:BenjaminiFormula} which gives
the Hausdorff dimension in  \cite{Benjamini09} or the
average box-counting dimension in \cite{Duplantier11}. 
In particular, unlike Hausdorff dimension, the almost sure box-counting dimension of $f(E)$
cannot be found simply in terms of the box-counting dimension of $E$ and the
random variable $W$ underlying the $f$.
One can easily construct a Cantor-like set $E$ of box and Hausdorff dimensions $1/(1+p)$ with the
almost sure box dimension of $f(E)$ as the solution in \eqref{eq:BenjaminiFormula},
see Corollary \ref{thm:homogeneous}. But the
set $E_{\ba(p)}$ with $\ba=(n^{-p})_{n}$ also has box dimension
$1/(1+p)$ with the box dimension of $f(E_{\ba(p)})$ given by  \eqref{dimgen}, so  $E$ and
$E_{\ba(p)}$ have the same box dimension but with their random images having different box
dimensions. Thus the structure of the set and not just its box-counting dimension determine the
image dimension.

\medskip
We obtain different dimension results for sets accumulating at 0 because we seek a balance between
the behaviour of products of the $W_\bi$ along the `stem' $\{0^k\}_{k\in\N}$, which grows
like $\exp\E(\log W)$ (a `geometric' mean), and that of the trees that
branch off this stem and grow like $\E(W)$ (an `arithmetic' mean). 
These different large deviation behaviours are exploited in the proofs. The stark difference in
these two behaviours was analysed in detail in \cite{Troscheit17} in a different context. 

On the other hand, homogeneous, or regular sets, have a structure resembling that of a tree that
grows geometrically and there is no `stem' that distorts this uniform behaviour.

Finally we remark that Theorems \ref{thm:dimea} and \ref{thm:dimgen} can be extended to critical
cascades in a similar fashion to our general bounds. We ommit details to avoid unneccesary
technicalities.

\subsection{Specific $W$ distributions}
The expressions for the  box-counting dimension in \eqref{dimgen} and the lower and upper bounds
above can be simplified or numerically estimated for particular distributions of $W$. Most often
considered is a log-normal distribution, and we also examine a two-point discrete distribution, as
was done for the Hausdorff dimension of images in \cite{Benjamini09}.

\subsubsection{Log-normal $W$}
Let $E_\ba$ be the set formed by the sequence $\ba = \ba(p)  \in S_p$, and let $W$ be log-normally distributed with
parameters $\mu, \sigma$, that is $W=\exp X$ where $X = N(\mu, \sigma^2)$. The condition that
$\E(W)=1$ requires $\mu=-\sigma^2/2$ and we can compute $\gamma=-\E(\log_2 W) =
-\mu/\log2=\sigma^2/\log 4$. The
standing condition that $\E(W\log_2 W)<1$ can be shown to be equivalent to $\sigma^2 < \log 4$.
Further, the conditions in \eqref{standassump} and \eqref{extraassump} can easily be checked.
Let $S_1(p)$ and $S_2(p)$ be the general lower and upper bound given by Theorems \ref{thm:lower} and
\ref{thm:upper}, respectively, for these $W$. Then,
\[
  S_1(p) = \frac{1}{(1+p)(1+\tfrac{\sigma^2}{\log 4})}.
\]
Noting that
\[
  \E(W^t)  = \exp\Big( \tfrac{\sigma^2}{2}t(t-1) \Big),
\]
we can calculate the upper bound since \eqref{upbound} becomes the quadratic
\[
  \bdd E_{\ba} - S_2(p) = \frac{\sigma^2}{\log 4}S_2(p)(1-S_2(p)).
\]

To compute the almost sure dimension of $f(E_{\ba})$,
first note that for $x\geq \gamma$ the infimum in the numerator of the dimension formula \eqref{dimgen}
is zero. For $x\in(0,\gamma)$ the infimum occurs at
$t_0$ where
\[
  0=\tfrac{\partial}{\partial t}\Big|_{t=t_0}x t + \log_2\E(W^t) 
  =x+(2t_0-1)\frac{\sigma^2}{\log 4}\ \ \text{ giving }\  \ t_0 = \frac12 \left(1-\frac{x}{\gamma}  \right)
\]
giving
\[
  \inf_{t>0}(xt+\log_2\E(W^t)) = \frac{x}{2}\left(1-\frac{x}{\gamma}\right)-\frac{\gamma}{4} 
  \left(1-\frac{x}{\gamma}  \right)\left(1+\frac{x}{\gamma}  \right)
  =\frac{x}{2}-\frac{x^2}{4\gamma}-\frac{\gamma}{4}
  =-\frac{\left( x-\gamma \right)^2}{4\gamma}
\]
for $x<\gamma$ and $0$ otherwise. Notice in particular that the infimum is clearly continuous at
$x=\gamma$.
We obtain
\[
  \bdd f(E_{\ba}) =\sup_{0<x<\gamma}\frac{1-(x-\gamma)^2/(4\gamma)}{1+x+(1+\gamma)p}
\]
Differentiating the right hand side with respect to $x$ gives
\[
  \frac{\gamma\left( \gamma+2p\left( 1+\gamma \right)-2 \right)-x^2-2x\left( 1+p+p\gamma
  \right)}{4\gamma\left( 1+p+x+p\gamma \right)^2}.
\]
Equating this with $0$ and solving for $x$ gives two solutions since the numerator is quadratic
and the denominator is non-zero for $0<x<\gamma$. Only one solution of the quadratic is positive so
\[
  \bdd f(E_{\ba}) = \frac{1-(x_0-\gamma)^2/(4\gamma)}{1+x_0+(1+\gamma)p},
\]
where
\[
  x_0 =\sqrt{(1+p+p\gamma)^2+2p\gamma+\gamma^2+2p\gamma^2-2\gamma}-p\gamma-p-1.
\]
Figure \ref{fig:comparison2} contains a plot of the almost sure dimension of $f(E_{\ba}(p))$ with $W$ being log-normally
distributed for parameter $\sigma=\log 4 - \tfrac{1}{100}$, chosen to give clearly visible separation between
the dimension and the general bounds. 
\begin{figure}[htb]
  \begin{center}
    \includegraphics[width = .49\textwidth]{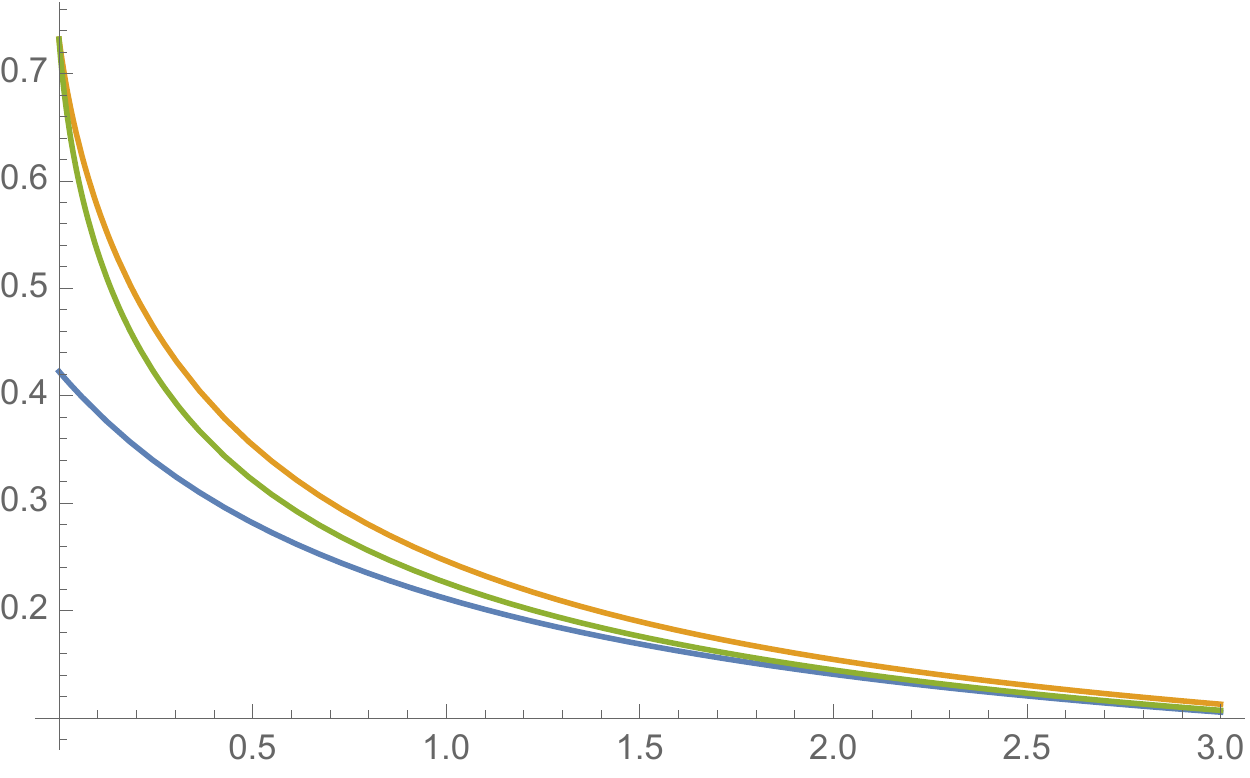}
    \hfill
    \includegraphics[width =0.49\textwidth]{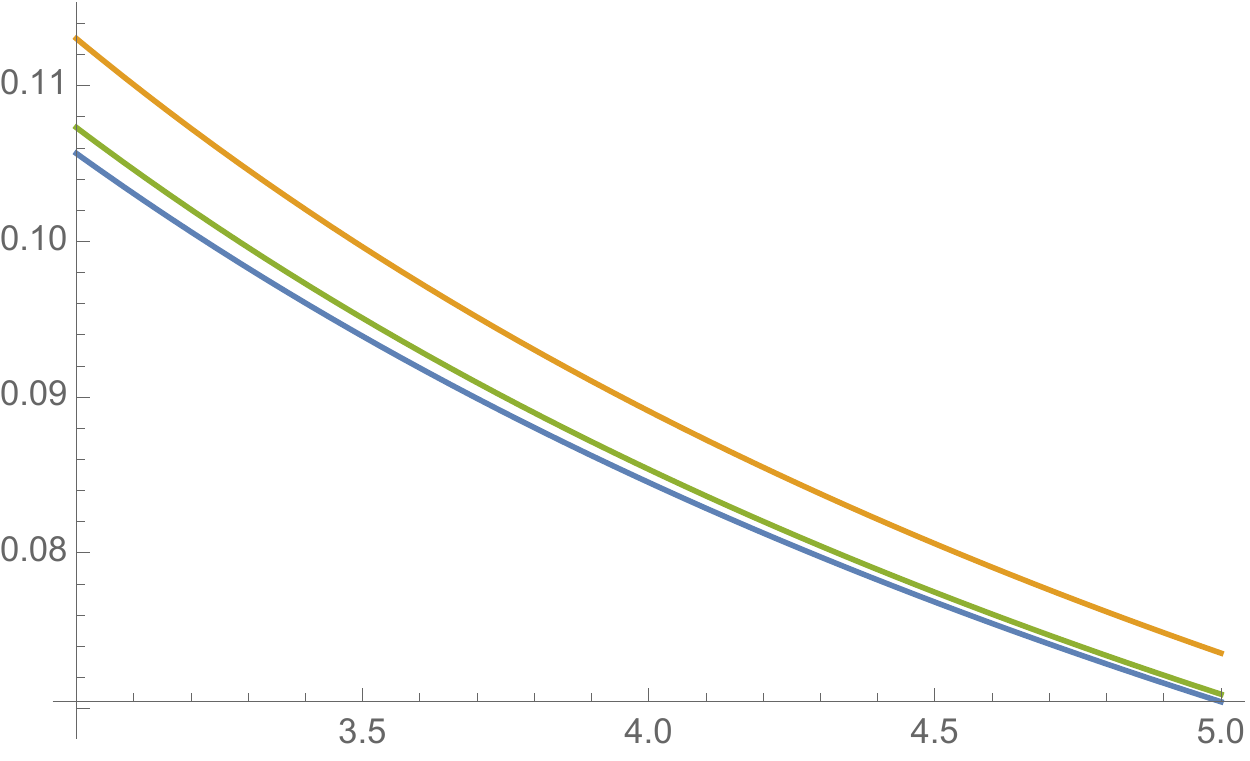}
    \caption{A plot of $S_1(p)\leq\overline\bdd f(E_{\ba})\leq S_2(p)$ for $0<p <
    3$ and $3<p<5$, where $W$ is a log-normal random variable with parameters $\sigma=\log4-1/100$.}
    \label{fig:comparison2}
  \end{center}
\end{figure}

\subsubsection{Discrete $W$}
Again, $E_\ba$ be the set formed by the sequence $\ba = \ba(p)  \in S_p$.
Fix a parameter $\xi\in (0,1)$ and let $W$ be the random variable satisfying
$\P(W=1-\xi)=1/2=\P(W=1+\xi)$. Clearly, $\E(W) = 1$ and our assumptions follow by the
boundedness of $W$. The geometric mean is $\gamma = -\E(\log_2 W) = -\log_2\sqrt{1-\xi^2}$
and Theorem \ref{thm:lower} gives the lower bound
\[
  S_1(p):=\frac{1}{(1+p)(1 -\log_2\sqrt{1-\xi^2})}.
\]
The upper bound $S_2(p)$ from  Theorem \ref{thm:upper} is implicitly given by
\[
%    2^{1/(1+p)}=\frac{2^{S_2(p)}}{\E(W^{S_2(p)})}
  2^{1/(1+p)}=\frac{2^{S_2(p)}}{\frac{1}{2}(1-\sigma)^{S_2(p)} + \frac{1}{2}(1-\sigma)^{S_2(p)}}.
\]

The functions  $S_1(p) \leq \ubd f(E_{\ba}) \leq S_2(p)$  for
$\xi=\tfrac{99}{100}$ are plotted in Figure \ref{fig:comparison}.
We were unable to find a closed form for $\bdd f(E_{\ba})$ from \eqref{dimgen} and  the figure was
produced computationally.

\begin{figure}[htb]
  \begin{center}
    \includegraphics[width = .49\textwidth]{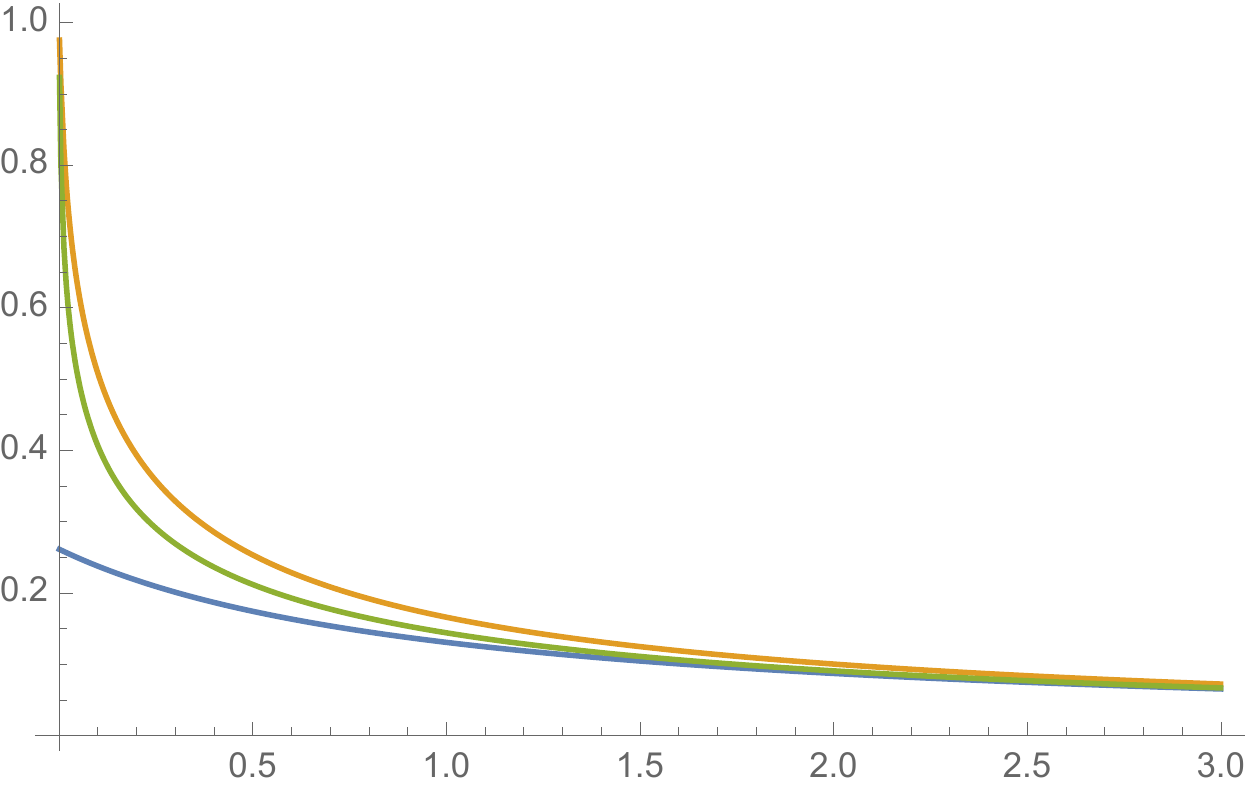}
    \hfill
    \includegraphics[width =0.49\textwidth]{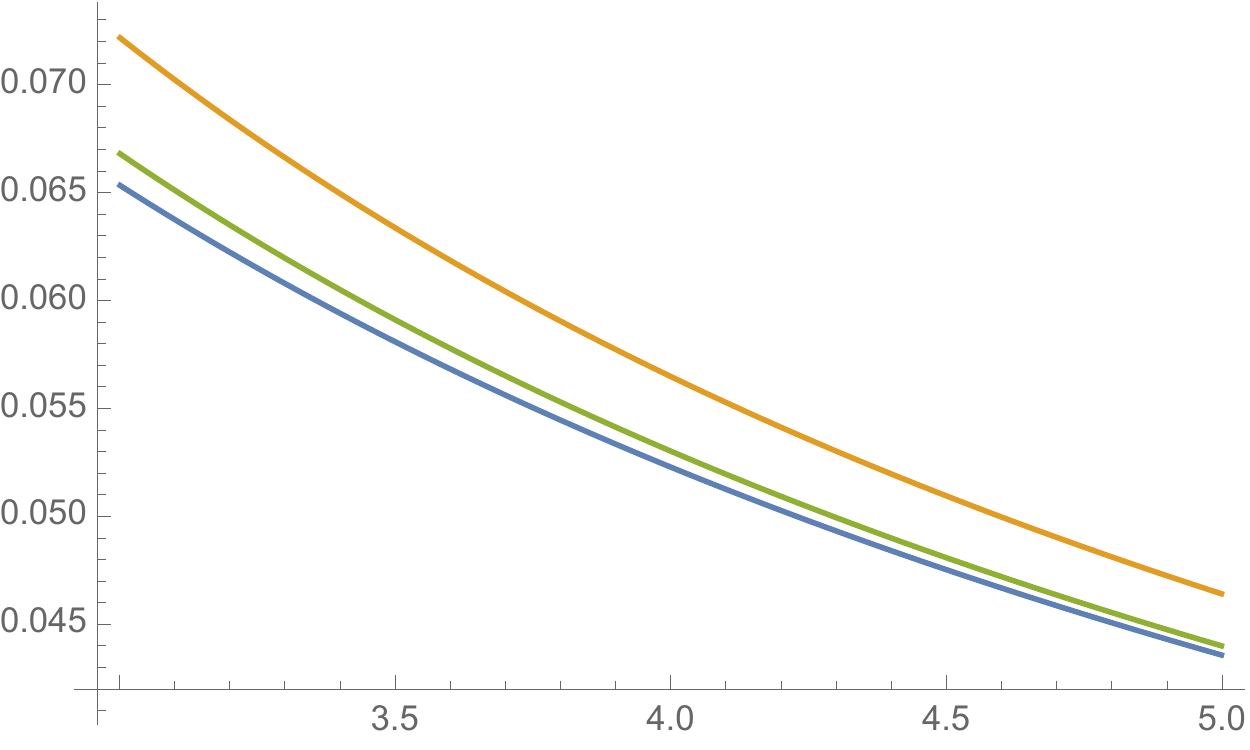}
    \caption{A plot of $S_1(p)\leq\overline\bdd f(E_{\ba})\leq S_2(p)$ for $0<p <
      3$ and $3<p<5$, where $W$ is a discrete random variable with $W=\tfrac{1}{100}$ and
    $W=\tfrac{199}{100}$ occurring with equal probability.}
    \label{fig:comparison}
  \end{center}
\end{figure}

\section{Proofs}

\subsection{General bounds}\label{genbounds}
In this section we prove Theorems \ref{thm:upper} and \ref{thm:lower} giving almost sure bounds for
$\ubd f(E)$ and $\lbd f(E)$ for a general set $E\subset [0,1]$.

\subsubsection{General upper bound}
We establish Theorem \ref{thm:upper} by estimating the expected number of intervals $I_\bi$ such
that $f(I_\bi)$ intersects $f(E)$ and $|f(I_\i)|\geq r$, to provide an almost sure bound for this
number which we  relate to the upper box-counting dimension of $f(E)$.

\begin{proof}[Proof of Theorem \ref{thm:upper}]
  First consider $\mu$ to be a subcritical cascade measure.
  Let $d>\ubd E$ and let $0<t\leq 1$ satisfy 
  \be\label{tdef}
  2^{-t}2^d\E(W^t)<1.
  \ee
  Let $k\geq 0$ and $0<r\leq 1$.  For each $I_\i\in {\mathcal I}_k$, Markov's inequality gives
  \begin{eqnarray}
    \P\{|f(I_\i)|\geq r\} &=& \P\{2^{-k}W_{i_1}W_{i_1,i_2}\cdots W_{\bi} L_{\bi}\geq r\}\nonumber\\
    &\leq& \E(2^{-kt}W_{i_1}^tW_{i_1,i_2}^t\cdots W_{\bi}^t L_{\bi}^t r^{-t})\nonumber\\
    &=& 2^{-kt}\E(W^t)^k\,\E(L^t)r^{-t}.\label{probr}
  \end{eqnarray}
  We estimate the expected number of dyadic intervals with image of length at least $r$. 
  For each $k\in \N$, let ${\mathcal J}_k$ be the set of intervals in ${\mathcal I}_k$ that
  intersect $E$ and let $N_{2^{-k}}(E)= \#({\mathcal J}_k)$ be the number of such intervals, so
  $N_{2^{-k}}(E)\leq 2^{dk}$ for all sufficiently large $k$.
  Let 
  $$A_k^r =\{\i : I_\i \in {\mathcal J}_k: |f(I_\i)|\geq r\}.$$
  From \eqref{probr}, the fact that $\E(L^t)\leq \E(L) =1$, and that  $\P\{|f(I_\i)|\geq r\} \leq 1$, 
  $$\E(\# A_k^r) \leq 2^{dk}\min\big\{1, 2^{-kt}\E(W^t)^k r^{-t}\big\}.$$
  Let $k_0$ be the least integer such that
  \be\label{kodef}
  2^{-t}\E(W^t) \leq 2^{-k_0t}\E(W^t)^{k_0} r^{-t}<1.
  \ee
  Then 
  \begin{eqnarray}
    \E\Big(\sum_{k=0}^\infty \# A_k^r\Big) & \leq& \sum_{k=0}^{\infty}2^{dk}\min\big\{1,
    2^{-kt}\E(W^t)^k r^{-t}\big\}\nonumber\\
    & \leq&  \sum_{k=0}^{k_0} 2^{dk} +  \sum_{k={k_0}+1}^\infty 2^{dk}2^{-kt}\E(W^t)^k r^{-t}\nonumber\\
    & \leq& c_1\,  2^{dk_0}\nonumber\\
    & \leq& c_1\,  (2^{t}\E(W^t)^{-1})^{k_0}\nonumber\\
    & \leq& c_1\,  r^{-t},\label{earkbound}
  \end{eqnarray}
  where we have used \eqref{tdef} and \eqref{kodef}, and where $c_1$ does not depend on $k\geq 0$ or $0<r\leq 1$.

  Note that, for $0<r<1$, the image set $f(E)$ is covered by the disjoint 
  intervals $\{f(I_\i)\}_{\i\in {\mathcal S}_r}$ where ${\mathcal S}_r =\{I_\i\in {\mathcal J}_k :
  |f(I_\i)| < r,  |f(I_{\i^-})| \geq r\}$, with
  $\i^- = i_1,\ldots,i_{k-1}$ if $\i = i_1,\ldots,i_{k}$. We denote by $N'_r(F)$ the minimal number of
  intervals of lengths at most $r$ that intersect the set $F$. Then
  \be
  N'_r(f(E))\ \leq \ \# {\mathcal S}_r \ \leq \  2  \sum_{k=0}^\infty \# A_k^r, \label{nrpr}
  \ee
  since each interval $f(I_\i)$ with $\i\in {\mathcal S}_r$ has a parent interval $f(I_{\i^-})$ with
  $|f(I_{\i^-})| \geq r$ with at most two such $f(I_\i)$ having a common parent interval.

  We now sum over a geometric sequence of  $r =2^{-n}$. Let $\epsilon>0$. 
  From \eqref{nrpr} and \eqref{earkbound}
  $$\E\big(N'_{2^{-n}}(f(E))\big) 2^{-nt-n\epsilon}\  \leq \ 2 c_3 2^{-n\epsilon},$$
  so
  $$\E\Big(\sum _{n=1}^\infty N'_{2^{-n}}(f(E)) 2^{-nt-n\epsilon}\Big)\  \leq \ 2 c_3 \sum
  _{n=1}^\infty2^{-n\epsilon} <\infty.$$
  Hence, almost surely, 
  $ N'_{2^{-n}}(f(E)) 2^{-nt-n\epsilon} $ is bounded in $n$, so from the definition of  box-counting
  dimension, noting that it is enough to take the limit through a geometric sequence $r =2^{-n}\to
  0$, we conclude that $\ubd f(E) \leq t+\epsilon$ for all $\epsilon>0$.
  Since $\epsilon$ is arbitrary $\ubd f(E) \leq t$. We may let $d\searrow d_E=\ubd E$  and
  correspondingly let $t \nearrow s$ with $t$ satisfying \eqref{tdef}, where $s$ is given by
  \eqref{upbound}, recalling that $t\mapsto 2^{t}\E(W^t)^{-1}$ is increasing and continuous. Thus
  almost surely $\ubd f(E) \leq s$ where $s$ satisfies  \eqref{upbound}.

  \vskip.5em
  If $\mu$ is the critical cascade measure, the proof follows similarly. We can first estimate
  \begin{eqnarray}
    \P\{|f(I_\i)|\geq r\} &=& \P\{\sqrt{k}\cdot2^{-k}W_{i_1}W_{i_1,i_2}\cdots W_{\bi} L_{\bi}\geq r\}\nonumber\\
    &\leq& \E(k^{t/2}\cdot2^{-kt}W_{i_1}^tW_{i_1,i_2}^t\cdots W_{\bi}^t L_{\bi}^t r^{-t})\nonumber\\
    &=& k^{t/2}2^{-kt}\E(W^t)^k\,\E(L^t)r^{-t}.\nonumber
  \end{eqnarray}
  Noting that $\E(L^t) < \infty$ for $t\in[0,1)$, see \cite[Theorem 1.5]{Hu09} or \cite[Equation
  (26)]{Barral14}, gives
  $$\E(\# A_k^r) \leq C 2^{dk}\min\big\{1, k^{t/2}2^{-kt}\E(W^t)^k r^{-t}\big\}
  \leq C k^{t/2}2^{d k} \min\big\{1, 2^{-kt}\E(W^t)^k r^{-t}\big\}$$
  for some constant $C>0$ and 
  one obtains an additional subexponential contribution to the expected covering number. The rest of
  the proof follows in much the same way and details are left to the reader.
\end{proof}

\subsubsection{General lower bound}
For the lower bound, Theorem \ref{thm:lower}, we note that, by the strong law of large numbers, 
\[
  \frac{1}{k}\log (W_{\ell_1}W_{\ell_2}\cdots W_{\ell_k})\to \E(\log W)
\]
almost surely, where the $W_{\ell_i}$ are independent with the distribution of $W$. This enables us to deduce that
a significant proportion of the intervals $f(I_\bi)$ that intersect $f(E)$ must be reasonably large.
Further, since we are taking logarithms we can ignore any subexponential growth which in particular means
that also
\[
  \frac{1}{k}\log (\sqrt{k} W_{\ell_1}W_{\ell_2}\cdots W_{\ell_k})\to \E(\log W)
\] 
almost surely.

We will use the following two lemmas.
\begin{lem}\label{thm:lowerlem}      
  Let $0<p\leq1$ and let $X_1,\ldots,X_n$ be events such that $\P(X_i)\geq p$ for all $1\leq i\leq n$. Let $0<\lambda <p$. Then
  \be
  \P\{\text{\rm at least } \lambda n\text{ \rm of the } X_i \text{ \rm occur}\} \geq \frac{p-\lambda}{1-\lambda}.\label{lem1}
  \ee 
\end{lem}
Note that there is no independence requirement on the $X_i$.

\begin{proof}
  Let $Y$ be the event $\{\text{\rm at least } \lambda n\text{ \rm of the } X_i \text{ \rm occur}\}$
  and let $\P(Y) = \rho$.  By the law of total expectation
  \begin{align*}
    pn&\leq \E(\# i : X_i  \text{ \rm occurs}) = \E(\# i : X_i  \text{ \rm occurs}| Y)\rho+\E(\# i : X_i  \text{ \rm occurs}|Y^c)(1-\rho)\\
    &\leq  n\rho + \lambda n(1-\rho).
  \end{align*}
  Hence  $p\leq \rho +\lambda(1-\rho)$ giving $\eqref{lem1}$.
\end{proof}

The following lemma can be derived from Hoeffding's inequality.

\begin{lem}\label{thm:hoeffding}
  Let $(X_i)$ be a sequence of i.i.d.\ binomial random variables with $\P(X_i=1) = p$
  and $\P(X_i=0)=1-p$.
  Then, 
  \[
    \P\bigg(\sum_{i=1}^N X_i \geq \tfrac12 p N\bigg) \geq 1-\exp\left( \tfrac12 p^2 N \right)
  \]
  and
  \[
    \P\bigg( \sum_{i=1}^N (1-X_i) \geq (1-\tfrac12 p) N \bigg) \leq \exp\left( -\tfrac12 p^2 N \right).
  \]
\end{lem}
\begin{proof}
  Hoeffding's inequality states that for any sequence of independent random variables $Y_i$ with $a_i \leq Y_i\leq b_i$ and for $t>0$,
  \[
    \P\bigg(\sum_{i=1}^N (Y_i - \E(Y_i)) \geq t\bigg) \leq \exp\bigg( -\frac{2t^2}{\sum_{i=1}^N(b_i-a_i)^2}
    \bigg).
  \]
  Thus,
  \begin{align*}
    \P\bigg(\sum_{i=1}^N X_i \geq \tfrac12 p N\bigg) &\geq \P\bigg(\sum_{i=1}^N (X_i-p) >
    \tfrac12 p N-pN\bigg)\\
    &=\P\bigg(\sum_{i=1}^N (X_i-\E(X_i)) >-
    \tfrac12 p N\bigg)\\
%    &=1-\P\bigg( \sum_{i=1}^N (X_i-\E(X_i)) \leq 
%    -\tfrac12 p N \bigg)\\
    &=1-\P\bigg( \sum_{i=1}^N (-X_i-\E(-X_i)) \geq 
    \tfrac12 p N \bigg) \\
    &\geq 1- \exp\bigg( -\frac{2( (1/2)p N)^2}{\sum_{i=1}^N 1} \bigg)
    = 1-\exp\big( -\tfrac12 p^2 N \big),
  \end{align*}
  where we have applied Hoeffding's inequality with $Y_i = -X_i$,$t=\tfrac12 p N$,$a_i=-1$ and $b_i=0$.

  For the second inequality we similarly obtain
  \begin{align*}
    \P\bigg( \sum_{i=1}^N (1-X_i) \geq (1-\tfrac12 p) N \bigg)
    &=
    \P\bigg( (1-p)N + \sum_{i=1}^N (-X_i)-\E(-X_i) \geq (1-\tfrac12 p) N \bigg)
    \\
    &=\P\bigg(  \sum_{i=1}^N (-X_i)-\E(-X_i) \geq \tfrac12 p N  \bigg)
    \\
    &\leq \exp\left( -\tfrac12p^2 N \right).\qedhere
  \end{align*}
\end{proof}

\begin{proof}[Proof of Theorem \ref{thm:lower}]

  Write $d= \ubd E$ and let $\epsilon >0$. 
  Then, for each $\i = i_1, i_2, \ldots \in \{0,1\}^\mathbb{N}$, by the strong law of large numbers,
  $\frac{1}{k}\log (2^{-k}W_{i_1}W_{i_1,i_2}\cdots W_{i_1,\ldots i_k})\to \E(\log W)-\log 2$ almost
  surely,  so there is some $k_0\in \mathbb{N}$ such that
  \begin{equation*}%\label{imest}
    \P\big\{2^{-k}W_{i_1}W_{i_1,i_2}\cdots W_{i_1,\ldots i_k} \geq 2^{-k(1-E(\log_2 W)+\e)} \big\} \geq {\textstyle\frac{3}{4}}
  \end{equation*}
  for all $k\geq k_0$. As  $L_{\bi}$ has the distribution of $L$, there exists $\tau>0$ such that
  $\P\{L_{\bi}\geq \tau\} =\P\{L\geq \tau\}\geq\frac{3}{4}$. Since  $|f(I_\i)| =
  2^{-k}W_{i_1}W_{i_1,i_2}\cdots W_{i_1,\ldots i_k}L_{\bi}$, and $L_{\bi}$ is independent of
  $\{W_{i_1},\ldots, W_{i_1,\ldots i_k}\}$,
  \begin{equation}\label{imest2}
    \P\big\{|f(I_\i)| \geq \tau2^{-k(1-E(\log_2 W)+\e)} \big\} \geq {\textstyle\frac{1}{2}}
  \end{equation}
  for each $\bi \in \{0,1\}^k$ if $k\geq k_0$.

  The same argument can be repeated for the critical case. Here, the strong law of large numbers
  gives $\frac{1}{k}\log (\sqrt{k}2^{-k}W_{i_1}W_{i_1,i_2}\cdots W_{i_1,\ldots i_k})\to \E(\log W)-\log 2$ almost
  surely and so for $k$ large enough,
  \begin{equation*}
    \P\big\{\sqrt{k}2^{-k}W_{i_1}W_{i_1,i_2}\cdots W_{i_1,\ldots i_k} \geq 2^{-k(1-E(\log_2 W)+\e)}
    \big\} \geq {\textstyle\frac{3}{4}}.
  \end{equation*}
  Again $L_i$ is equal to $L$ in distribution and there exists $\tau>0$ such that
  $\P\{L\geq\tau\}\geq \tfrac34$.
  We can now conclude that \eqref{imest2} also holds in the critical case.

  For each $k\in \N$, let ${\mathcal J}_k$ be the set of intervals in ${\mathcal I}_k$ that  
  intersect $E$, and let $\#({\mathcal J}_k)$ be the number of such intervals. By the definition
  of upper box-counting dimension $\#({\mathcal J}_k)\geq 2^{k(d-\epsilon)}$ for infinitely many $k$; write
  $K$ for this infinite set of $k \geq k_0$.
  Applying Lemma \ref{thm:lowerlem} to the intervals $I_\i\in{\mathcal J}_k$, taking $p=\frac{1}{2}$
  and $\lambda = \frac{1}{4}$, 
  \be\label{secbound}
  \P\big\{|f(I_\i)| \geq \tau2^{-k(1-E(\log_2 W)+\e)}  \text{ \rm for at least }
  {\textstyle\frac{1}{4} }2^{k(d-\epsilon)} \text{ \rm of the } I_\i \in {\mathcal J}_k\big\} \geq
  {\textstyle\frac{1}{3}},
  \ee
  for all $k \in K$.

  Let $N'_r(F)$ be the maximum number of disjoint intervals of lengths at least $r$ that intersect a set  $F$.
  Write $r_k = 2^{-k(1-E(\log_2 W)+\e)}$ for each $k\in \mathbb{N}$. From \eqref{secbound}, 
  $N'_{r_k}(f(E)) \geq {\textstyle\frac{1}{4} }2^{k(d-\epsilon)}$ with probability at least
  $\frac{1}{3}$ for each $k \in K$, so with probability at least $\frac{1}{3}$ it holds for
  infinitely many $k \in K$. It is easy to see that an equivalent definition of upper box-counting
  dimension is given by
  $\ubd F = \varlimsup_{r\to 0} \log_2 N'_r(F) /\log_2 (1/r)$. It is enough to evaluate this limit
  along the geometric sequence  $r=r_k$, so 
  $$
  \ubd f(E) = \varlimsup_{k\to \infty} \frac{\log_2 N'_{r_k}(F)}{-\log_2 r_k} \geq
  \frac{(d-\epsilon)}{(1-\E(\log_2 W) +\epsilon)}, 
  $$
  with probability at least $\frac{1}{3}$, and therefore with probability 1, since $\ubd f(E)\geq s$ is
  a tail event for all $s$. Since $\epsilon >0$ is arbitrary, \eqref{lowerbound} follows. 
  \medskip

  For the lower box dimensions for subcritical cascades, 
  we let $d= \lbd E$, which we may assume to be positive, and
  $0<\epsilon <d$. We need an estimate on the rate of convergence in the laws of large numbers: if
  $\E(|X|^p)<\infty$ for some $p>2$ then
  \be\label{BaumKatz}
  \sum_{k=1}^\infty \P\Big\{\Big|\sum_{i=1}^k X_i - k\mu\Big|>k\epsilon\Big\} <\infty;
  \ee
  this follows, for example, from estimates of Baum and Katz (taking $t=p$ and $r=2$ in \cite[Theorem 3(b)]{BK65}).
  For $\bi = i_1, i_2,\ldots\in \{0,1\}^\mathbb{N}$ write
  $$ P_k = \P\big\{2^{-k}W_{i_1}W_{i_1,i_2}\cdots W_{i_1,\ldots i_k} < 2^{-k(1-E(\log_2 W)+\e)} \big\}= 
  \P\Big\{\sum_{i=1}^k  \log_2 W_{\bi|k} - k \E(\log_2 W)< - k\e\Big\};$$
  noting that $P_k$ is independent of $\bi$.
  By \eqref{BaumKatz}  $\sum_{k=1}^\infty P_k <\infty$. For each $\bi\in \{0,1\}^k$ let $E_\bi$ be the event
  $$ E_\bi = \big\{2^{-k}W_{i_1}W_{i_1,i_2}\cdots W_{i_1,\ldots, i_k} \geq 2^{-k(1-E(\log_2 W)+\e)} \big\},$$
  so $\P(E_\bi) = 1-P_k$.  

  For each $k\in \N$, let ${\mathcal J}_k$ be the set of intervals in ${\mathcal I}_k$ that  
  intersect $E$, so there is a number $k_0$ such that if $k\geq k_0$ then $\#(\mathcal{J}_k)\geq 2^{k(d-\epsilon)}$. 
  Fixing $k\geq k_0$, 
%  $$ E_\bi = \Big\{\sum_{i=1}^k  \log_2 W_{\bi|i} - k \E(\log_2 W)\geq - k^{2/p}\Big\}$$
  let $\mathcal{E}_k = \{\bi \in \mathcal{J}_k: E_\bi \text{ occurs}\}$,  which depends only on $\{W_\bi : |\bi| \leq k\}$.
  By Lemma \ref{thm:lowerlem},
  $$\P\big\{\#(\mathcal{E}_k)\,  \geq\,{\textstyle \frac{1}{2}}2^{k(d-\epsilon)}\big\} \geq
  \frac{1-P_k -\textstyle{\frac{1}{2}}}{1-\textstyle{\frac{1}{2}}}\ =\ 1-2 P_k.
  $$
 %$$
%\P\big\{E_\bi \text{ for at least }  
%\textstyle{\frac{1}{2}}2^{k(d-\epsilon)} \text{ of  }\bi \in \mathcal{I}_k\big\}  \geq \frac{1-P_k
%-\textstyle{\frac{1}{2}}}{1-\textstyle{\frac{1}{2}}} = 1-2 P_k.
%$$ 
  The random variables $\{L_\bi: \bi \in \mathcal{I}_k\}$ are independent of $\{W_\bi : |\bi| \leq
  k\}$ and of each other. Let $\P\{L_\bi\geq 1\}=\P\{L\geq 1\}=p>0$. Conditional on
  $\big\{\#(\mathcal{E}_k)  \geq {\textstyle \frac{1}{2}}2^{k(d-\epsilon)}\big\}$, a standard
  binomial distribution estimate, which follows from Hoeffding's inequality (see Lemma
  \ref{thm:hoeffding}), gives that 
  $$
  \P\big\{\#( \bi \in \mathcal{E}_k :   L_\bi\geq 1)\geq {\textstyle \frac{1}{2}}p\,\#(\mathcal{E}_k)\big\} \geq
  1-\exp\big(-{\textstyle \frac{1}{2}}p^2\#(\mathcal{E}_k) \big) 
  \geq 1-\exp\big(-{\textstyle \frac{1}{4}}p^2 2^{k(d-\epsilon)} \big).
  $$
  Hence, unconditionally, for each $k$,
  \begin{align*}
    \P\Big\{\#&\big( \bi \in \mathcal{I}_k :|f(I_\bi)|\geq 2^{-k(1-\E(\log_2 W)+\e) }\big)\geq {\textstyle \frac{1}{4}}p2^{k(d-\epsilon)}\Big\}\\
    &\geq  \P\Big\{\#\big( \bi \in \mathcal{I}_k : 2^{-k}W_{i_1}W_{i_1,i_2}\cdots W_{\bi}\geq 2^{-k(1-\E(\log_2 W)+\e)}\text{ and }  L_\bi\geq 1 \big)\geq {\textstyle \frac{1}{4}}p2^{k(d-\epsilon)}\Big\}\\
    & \geq (1-2 P_k)\big(1-\exp\big(-{\textstyle \frac{1}{4}}p^2 2^{k(d-\epsilon)} \big)\big)\\
    & \geq 1-2 P_k-\exp\big(-{\textstyle \frac{1}{4}}p^2 2^{k(d-\epsilon)} \big).
  \end{align*}
  Since $\sum_{k=1}^\infty 2P_k <\infty$ and $\sum_{k=1}^\infty \exp\big(-{\textstyle
  \frac{1}{4}}p^2 2^{k(d-\epsilon)} \big) <\infty$, the Borel-Cantelli lemma implies that, with probability one, 
  $$\#\big\{ \bi \in \mathcal{I}_k :|f(I_\bi)|\geq 2^{-k(1-\E(\log_2 W)+\e)}\big\}\geq {\textstyle
  \frac{1}{4}}p2^{k(d-\epsilon)}$$
  for all sufficiently large $k$.
  As in the upper dimension part, but taking lower limits, it follows that $\lbd f(E) \geq
  {(d-\epsilon)}{(1-\E(\log_2 W)+\e )}$ for all $\epsilon >0$, giving \eqref{lowerbound2}.

  For the lower box dimensions and critical cascades we note that
  $$ P_k = \P\big\{\sqrt{k}\cdot 2^{-k}W_{i_1}W_{i_1,i_2}\cdots W_{i_1,\ldots i_k} < \sqrt{k}\cdot
  2^{-k(1-E(\log_2 W)+\e)} \big\}.$$
  Following the same argument as above with the additional $\sqrt{k}$ term we conclude that 
  $$\#\big\{ \bi \in \mathcal{I}_k :|f(I_\bi)|\geq \sqrt{k}\cdot2^{-k(1-\E(\log_2 W)+\e)}\big\}\geq {\textstyle
  \frac{1}{4}}p2^{k(d-\epsilon)}$$
  for sufficiently large $k$.
  Again, taking lower limits and noting that $\tfrac{1}{k}\log\sqrt{k}\to 0$ we get the required
  lower bound for critical cascades. 
\end{proof}

\subsubsection{Asymptotic behaviour}
\begin{proof}[Proof of Proposition \ref{thm:asymptotics}]
  Solving \eqref{eq:lower} for $d$ and substituting in \eqref{eq:upper} gives
  \[
    s_2(1-\E(\log_2 W))  = s_1 - \log_2\E(W^{s_1}).
  \]
  Rearranging gives
  \[
    \frac{s_1}{s_2} = \frac{1-\E(\log_2 W)}{1-\log_2\E(W^{s_1})^{1/s_1}}
    =\frac{\log 2-\E(\log W)}{\log 2 - \log\E(W^{s_1})^{1/s_1}}.
  \]
  Note that $s_1,s_2 \to 0$ as $d\to 0$. 
  Recall that our assumptions imply $\E(\log W)<\log 2$ and $\E(W^t)<\infty$ for all $t\in[0,1]$.
  It is well-known that the power means converge to the geometric mean, i.e.\ $\E(W^{s_1})^{1/s_1}\to
  \exp\E(\log W)$. Combining this with the above means that $s_1/s_2\to 1$ as required.
\end{proof}

\subsection{Box dimension of images of decreasing sequences}%\label{maindim}
We now proceed to the substantial proof of Theorems \ref{thm:dimea} from which we easily deduce
Theorem \ref{thm:dimgen}. First, the following lemma notes some properties of the expressions that
occur in \eqref{dimeal} and \eqref{dimgen}, in particular it follows that they are continuous in
$\alpha$ and $p$ respectively (for example, the right hand side of \eqref{dimeal} is
$\phi\big((1+\gamma)/\alpha)\big)$ with $\phi$ as in  \eqref{phidef}).

\begin{lem}\label{thm:monotone}
  (a) For $x\geq 0$ let
  \[
    \psi(x):= \inf_{t\geq 0}\left( xt+\log_2\E(W^t) \right).
  \]
  If $x\geq \gamma $ this infimum is attained at $t=0$. If $x\in (0,\gamma)$ the infimium is attained
  at $t\in (0,1)$.
  Furthermore $\psi(x)$ is continuous for $x\geq 0$.

  (b) For $\beta \geq 0$ let
  \be\label{phidef}
  \phi(\beta) = \sup_{x>0} \frac{1+\inf_{t>0}\left( xt+\log_2\E(W^t) \right)}{1+x+\beta}.
  \ee
  Then $\phi$ is strictly decreasing and continuous in $\beta$. 
\end{lem}

\begin{proof}
  (a) Let $g_x(t) =  xt+\log_2\E(W^t)$ for $x\geq 0$ and $t\geq 0$. Then $g_x''(t)>0$ by
  \eqref{derivratio} so $g_x$ is a strictly convex function. Also  $g'_x(t) =  x+\frac{\E(W^t \log_2
  W)}{\E(W^t)}$, so in particular, $g'_x(0) = x+ \E(\log_2 W)=x-\gamma$ and $g'_x(1) = x+ \E(W
  \log_2 W)> x+\E(W)\log_2 \E(W)= x>0$, by Jensen's inequality and  that $W$ is not almost surely
  constant, so the conclusions in (a) on the infimum follows. The function $\psi$ is continuous for
  $x\geq 0$ since it is the Legendre transform of the twice continuously differentiable strictly
  convex function $\log_2\E(W^t)$.
  \medskip

  (b) Now consider the function 
  $$\eta(x,\beta)  =  \frac{1+\psi(x) }{1+x+\beta}, \qquad (x\in [0,\gamma], \beta \geq 0), $$
  which is continuous for $(x,\beta) \in [0,\infty) \times [0,\gamma]$, and note that $\phi(\beta) =
  \sup_{x\in [0,\gamma]}\eta(x,\beta)$. Since the supremum in $\phi(\beta)$ is over a bounded
  interval, it is an exercise in basic analysis to see that  $\phi$ is continuous in $\beta$ and
  that, since $\eta(x,\beta)$ is strictly decreasing in $\beta$ for each $x$, $\phi$ is strictly
  decreasing.
\end{proof}

\subsubsection{Upper bound for $\dim_B f(E^\alpha)$}

Throughout this section, the distribution of $W$, and so $\gamma = -\E(\log_2 W)$, are fixed, as is $\alpha>0$.

First we bound the expected number of intervals of length at most $r$ needed to cover the part of
$f(E^\alpha \cap  [2^{-k},2^{-k+1}])$ by bounding the expected number of dyadic intervals $I_\bi$ in
$[2^{-k},2^{-k+1}]$ that intersect $E$ such that $|f(I_\bi)| \geq r$.

\begin{lem}\label{thm:newprelim}
  Let $0< \e<\gamma$. Let $k\in \mathbb{N}$ and suppose that $W_0 W_{00}\ldots W_{0^{k-1}} \leq
  a2^{-(k-1)(\gamma-\e)}$ for some $a>0$. Then for all $0<t<1$, there exists  $c_t >0$ such that 
  \begin{equation}\label{enr}
    \E\big(N_r(f(E^\alpha\cap [2^{-k},2^{-k+1}]))\big) \leq c_t r^{-t}2^{-kt(1+\gamma-\e)} \big(2^{1-t}\E(W^t)\big)^{\alpha k} +k
  \end{equation}
  for all $0<r<1$. The numbers $c_t$ may be taken to vary continuously in $t\in (0,1)$ and do not depend on $\e, k$ or $r$. 
\end{lem}

\begin{proof}
  We bound from above the expected number of dyadic intervals $I_\bi$ which intersect $E^\alpha\cap
  [2^{-k},2^{-k+1}]$ such that $|f(I_\bi)| \geq r$.
  We split these intervals into three types.
  \medskip

  \noindent (a) There are $k$ intervals $I_\emptyset, I_{0}, I_{00},\ldots,I_{0^{k-1}}$ which cover
  $E^\alpha\cap [2^{-k},2^{-k+1}]$ to give the right-hand term of \eqref{enr}.
  \medskip

  \noindent (b) Consider $I_\bi$ of the form $\bi = 0^{k-1}1\bj$ where $\bj \in \{0,1\}^{j}$ and
  $0\leq j = |\bj| \leq \lfloor \alpha k\rfloor$. Then
  \begin{align}
    \P\big(|f(I_\bi)| \geq r\big)
    &= \P\big(2^{-(k+j)}W_0 W_{00}\ldots W_{0^{k-1}}W_{0^{k-1}1}W_{0^{k-1}1j_1}\ldots W_{0^{k-1}1j_1\ldots j_j}L_\i \geq r\big)\nonumber\\
    &\leq \P\big(2^{-(k+j)}a2^{-(k-1)(\gamma-\e)}W_{0^{k-1}1}W_{0^{k-1}1j_1}\ldots W_{0^{k-1}1j_1\ldots j_j}L_\i \geq r\big)\nonumber\\ 
    &\leq  a^t r^{-t}2^{-(k+j)t}2^{-(k-1)(\gamma-\e)t}\E\big(W_{0^{k-1}1}^t W_{0^{k-1}1j_1}^t\ldots W_{0^{k-1}1j_1\ldots j_j}^t L_\i^t\big)\label{pfiia}\\
    &=  \big(a^t 2^{(\gamma-\e)t} \E(W^t) \E( L^t) \big)\,r^{-t} 2^{-kt(1+\gamma-\e)} \big(2^{-t} \E\big(W^t )\big)^j \label{pfii}
  \end{align}
  where we have raised the condition to power $t$  and used Markov's inequality and the independence of the $W$s and $L_\i$.
  Hence for each $0<j \leq \lfloor \alpha k\rfloor$,
  \begin{align}
    \E\big(\# \bi: \bi =  0^{k-1}1\bj ,\,  |\bj|= j \text{ and }  |f(I_\bi)| \geq r \big) &= 2^j \P\big( |f(I_\bi)| \geq r \big)\nonumber\\
    &\leq b_t\, r^{-t} 2^{-kt(1+\gamma-\e)}\big(2^{1-t} \E\big(W^t )\big)^j  \label{enoi}
  \end{align}
  using \eqref{pfii}, where $b_t = a^t 2^{\gamma t} \E(W^t) \E( L^t)$. Since $1< 2^{1-t} \E\big(W^t
  )<2$ for $t\in (0,1)$, we can sum \eqref{enoi} over $0\leq j \leq \lfloor \alpha k\rfloor$ to get 
  \begin{equation}\label{uptoak}
    \E\big(\# \bi: \bi =  0^{k-1}1\bj ,\, 0\leq |\bj| \leq \lfloor \alpha k\rfloor \text{ and }  |f(I_\bi)| \geq r \big) 
    \ \leq\  b'_t\, r^{-t} 2^{-kt(1+\gamma-\e)}\big(2^{1-t} \E\big(W^t )\big)^{\lfloor \alpha k\rfloor},
  \end{equation}
  where  $b'_t = b_t/\big(1- (2^{t-1} \E(W^t )^{-1})\big)$. Note that $b'_t$ is continuous on $(0,1)$.
  \medskip

  \noindent (c) Now consider $I_\bi$ of the form $\bi = 0^{k-1}1\bj0^\ell $ where $\bj \in
  \{0,1\}^{\lfloor \alpha k\rfloor}$ and $1\leq \ell<\infty$. Then, as in case (b) but including the
  terms for levels $k+\lfloor \alpha k\rfloor +\ell$, we get, just as in \eqref{pfiia},
  \begin{align}
    \P\big(|f(I_\bi)| \geq r\big)
    &\leq  a^t r^{-t}2^{-(k+\lfloor \alpha k\rfloor +\ell)t }2^{-(k-1)(\gamma-\e)t}\nonumber\\
    &\hspace{2cm}\cdot\E\big(W_{0^{k-1}1}^t W_{0^{k-1}1j_1}^t\ldots W_{0^{k-1}1\bj}^tW_{0^{k-1}1\bj 0}^tW_{0^{k-1}1\bj 00}^t\ldots W_{0^{k-1}1\bj 0^\ell}^t L_\i^t\big)\nonumber\\
    &=  \big(a^t 2^{(\gamma-\e)t} \E(W^t) \E( L^t) \big)\,r^{-t} 2^{-kt(1+\gamma-\e)} \big(2^{-t} \E\big(W^t )\big)^{\lfloor \alpha k\rfloor +\ell} \label{pfiic}.
  \end{align}
  Hence for each $1\leq \ell<\infty$,
  \begin{align}
    \E\big(\# \bi : \bi= 0^{k-1}1\bj0^\ell ,\,  |\bj| =  \lfloor \alpha k\rfloor & \text{ and }  |f(I_\bi)| \geq r \big) = 2^{\lfloor \alpha k\rfloor} \P\big( |f(I_\bi)| \geq r \big)\nonumber\\
    &\leq b_t\, r^{-t} 2^{\lfloor \alpha k\rfloor} 2^{-kt(1+\gamma-\e)}\big(2^{-t} \E\big(W^t )\big)^{\lfloor \alpha k\rfloor +\ell} \label{enoib}
  \end{align}
  using \eqref{pfiic}, where $b_t = a^t 2^{\gamma t} \E(W^t) \E( L^t)$ as above. Since
  $\frac{1}{2}\leq 2^{-t} \E\big(W^t )<1$ we can sum \eqref{enoib} over $1\leq \ell<\infty$ to get 
  \begin{align}
    \E\big(\# \bi : \bi= 0^{k-1}1\bj0^\ell ,\,  |\bj| =  \lfloor \alpha k\rfloor,\, &\ell\geq 1  \text{ and }  |f(I_\bi)| \geq r \big)\nonumber\\ 
    &\leq b_t\, r^{-t} 2^{\lfloor \alpha k\rfloor} 2^{-kt(1+\gamma-\e)}\big(2^{-t} \E\big(W^t )\big)^{\lfloor \alpha k\rfloor +1}\big/  \big( 1 - 2^{-t} \E\big(W^t )\big)\nonumber\\
    &\leq b_t^{''}\, r^{-t} 2^{-kt(1+\gamma-\e)}\big(2^{1-t} \E\big(W^t )\big)^{\lfloor \alpha k\rfloor}\label{enoi2}
  \end{align}
  where $b_t^{''} = b_t(2^{-t} \E\big(W^t )) / \big( 1 - 2^{-t} \E\big(W^t )\big)$ is continuous in $t$. 
  \medskip

  For $0<r<1$,  let $\mathcal{J}^{(r)}$ be the collection of all intervals $I_\bi$ of the form
  considered in (a),(b),(c) above that intersect $E^\alpha$ and such that $|f(I_{\bi^-})| \geq r$
  and $|f(I_{\bi})| < r$, where if $\bi = i_1 i_2\ldots i_{j}$ then $\bi^- = i_1 i_2\ldots i_{j-1}$,
  so the intervals $f(I_\bi)$ with $\bi\in \mathcal{J}^{(r)}$ have length at most $r$ and cover
  $f(E^\alpha\cap [2^{-k},2^{-k+1}])$. Each  $I_\bi \in \mathcal{J}^{(r)}$ has a `parent' interval
  $I_\bi^-$ with at most two intervals in $\mathcal{J}^{(r)}$ having a common parent interval. These
  parent intervals have $|f(I_{\bi^-})| \geq r$ and are included in those counted in (a),(b),(c) so
  $N_r(f(E^\alpha\cap [2^{-k},2^{-k+1}]))$ is bounded above by twice this number of intervals.

  Hence, combining (a), \eqref{uptoak} and \eqref{enoi2} we obtain \eqref{enr}, where
  $c_t=2\max\{b_t', b_t^{''}\}$ is continuous on $(0,1)$ and we can replace $\lfloor \alpha
  k\rfloor$ by $\alpha k$.
\end{proof}

By writing $r$ in an appropriate form relative to $2^{-k}$, we can bound the expectation in the
previous Lemma by $r$ raised to a suitable exponent. Note that in the following lemma we have to
work with the infimum over $[t_1,t_2]$ where $0<t_1<t_2<1$ in order to get a uniform constant
$c(t_1,t_2)$. At the end of the proof of Proposition \ref{propnew} we show that the infimum can be
taken over $t>0$.

\begin{lem}\label{thm:newprelim2}
  Let $0<\e<\gamma$. Let $k\in \mathbb{N}$ and suppose that $W_0 W_{00}\ldots W_{0^{k-1}} \leq
  a2^{-(k-1)(\gamma-\e)}$ for some $a>0$. Then for all $0<t_1<t_2<1$, there exists  $c(t_1,t_2) >0$,
  independent of $k,r$ and $\e$, such that,  provided that $t_2(\e):=1/(1+(1+\gamma-\e)/\alpha)<t_2<1$, 
  \begin{equation}\label{enr2}
    \E\big(N_r(f(E^\alpha\cap [2^{-k},2^{-k+1}]))\big) \leq c(t_1,t_2) r^{-\phi(t_1,t_2,\e)} +k
  \end{equation}
  for all $0<r<1$, where
  $$\phi(t_1,t_2,\e)=  \sup_{x>0} \frac{1+\inf_{t\in
  [t_1,t_2]}\left(xt+\log_2\mathbb{E}(W^t)\right)}{1+x+(1+\gamma-\e)/\alpha}.$$
\end{lem}

\begin{proof}
  In Lemma \ref{thm:newprelim} $c_t$ is continuous and positive on $(0,1)$, so let $c(t_1,t_2) = \sup_{t\in [t_1,t_2]} c_t>0$.
  For $0<r<1$ and $k \in \mathbb{N}$ define $x_k(r)> -1-(1+\gamma-\e)/\alpha $ by  
  \begin{equation}\label{xkrdef}
    r= 2^{-k(\alpha(1+x_k(r))+(1+\gamma-\e))}.
  \end{equation}
  We bound the right hand side of \eqref{enr} using \eqref{xkrdef}. For $t\in [t_1,t_2]$,
  \begin{align*}
    \log_2\big( r^{-t}&2^{-kt(1+\gamma-\e)}\big(2^{1-t}\E(W^t)\big)^{\alpha k}\big)\\
    &= 
    \log_2( r^{-t}) -kt(1+\gamma-\e) + \alpha k (1-t +\log_2 \E(W^t))\\
    & = kt\big(\alpha(1+x_k(r))+(1+\gamma-\e)\big) -kt\big(1+\gamma-\e\big) + \alpha k \big(1-t +\log_2 \E(W^t)\big)\\
    & = \alpha k\big(1+x_k(r)t +\log_2 \E(W^t)\big)
%    & \leq \alpha k\big(1+\inf_{t\in [t_1,t_2]}(x_k(r)t +\log_2 \E(W^t))\big).
  \end{align*}
  Changing the base of logarithms to $1/r$ and taking the infimum over $t\in [t_1,t_2]$,
  \begin{align*}
    \log_{1/r}\Big(\inf_{t\in [t_1,t_2]}&\big( r^{-t}2^{-kt(1+\gamma-\e)}\big(2^{1-t}\E(W^t)\big)^{\alpha k}\big)\Big)\\
    & \leq \alpha k\big(1+\inf_{t\in [t_1,t_2]}(x_k(r)t +\log_2 \E(W^t))\big)\big/\big(k(\alpha(1+x_k(r))+(1+\gamma-\e))\big)\\
    &=  \big(1+\inf_{t\in [t_1,t_2]}   (x_k(r)t +\log_2\E(W^t))\big)\big/\big(1+x_k(r)+(1+\gamma-\e)/\alpha\big)\\
    & \leq \phi(t_1,t_2,\e).
  \end{align*}
  Inequality \eqref{enr2} now follows from \eqref{enr} by taking the supremum over $x\equiv x_k(r)>  -1-(1+\gamma-\e)/\alpha$. If $x\leq 0$,
  $$ \frac{1+\inf_{t\in [t_1,t_2]}   (xt +\log_2\E(W^t))}{1+x+(1+\gamma-\e)/\alpha}
  \leq \frac{1+xt_2 +\log_2\E(W^{t_2})}{1+x+(1+\gamma-\e)/\alpha} \leq \frac{1+0t_2 +\log_2\E(W^{t_2})}{1+0+(1+\gamma-\e)/\alpha},$$
  since, by calculus, the middle term is increasing in $x$ for  $-1-(1+\gamma-\e)/\alpha< x\leq 0$,
  provided that  $t_2(\e)<t_2<1$, so it is enough to take the supremum over $x>0$.
\end{proof}

%\begin{align*}
%\E\big(N_r(f(E_\alpha\cap [2^{-k},2^{-k+1}]))\big)& 
% \leq c(t_1,t_2)2^{kt(\alpha(1+x_k(r))+(1+\gamma-\e))}2^{-kt(1+\gamma-\e)} \big(2^{1-t}\E(W^t)\big)^{ \alpha
%  k} +k\\
%&\leq c(t_1,t_2) 2^{\alpha k(1+x_k(r)t +\log_2\E(W^t))}+k\\
%&\leq c(t_1,t_2) r^{-\alpha (1+x_k(r)t +\log_2\E(W^t))/(\alpha(1+x_k(r))+(1+\gamma-\e))}+k\\
%&\leq c(t_1,t_2) r^{-\alpha (1+\inf_{t\in [t_1,t_2]}   (x_k(r)t +\log_2\E(W^t))/(\alpha(1+x_k(r))+(1+\gamma-\e))}+k\\
%&\leq c(t_1,t_2) r^{-\alpha (1+\inf_{t\in [t_1,t_2]}   (x_k(r)t +\log_2\E(W^t))/(\alpha(1+x_k(r))+(1+\gamma-\e))}+k\\
%&\leq c(t_1,t_2) r^{-\sup_{x>0}[ (1+\inf_{t\in [t_1,t_2]}   (xt +\log_2\E(W^t))/((1+x)+(1+\gamma-\e)/\alpha)]}+k
%\end{align*}

It remains to sum the estimates in Lemma \ref{thm:newprelim2} over $1\leq k \leq K$ for an
appropriate $K$ and make a basic estimate to cover $f(E^\alpha\cap [0,2^{-K}]))$. The Borel-Cantelli
lemma leads to a suitable bound for $N_r(f(E^\alpha))$ for all sufficiently small $r$, and finally
we note that the infimum can be taken over $t>0$.

\begin{prop}\label{propnew}
  Let $\alpha>0$. Under the assumptions in Theorem \ref{thm:dimea}, but without the need for
  \eqref{extraassump}, almost surely,
  \begin{equation}\label{updimallt}
    \overline{\dim}_B(f(E^\alpha)) \leq 
    \sup_{x>0} \frac{1+\inf_{t>0}\left(xt+\log_2\mathbb{E}(W^t)\right)}{1+x+(1+\gamma)/\alpha}.
  \end{equation} 
\end{prop} 

\begin{proof}

  Let $0<\e<\gamma$ and let $0<t_1<t_2<1$ with $t_2(\e)<t_2$, where $t_2(\e)$ is as in Lemma \ref{thm:newprelim2}. By the strong law of large numbers, $(W_0 W_{00}\ldots
  W_{0^{k}})^{1/k} \to 2^\gamma$ as $k\to\infty$, so almost surely there exists a random number $A>0$ such that 
  $W_0 W_{00}\ldots W_{0^{k}} \leq A\, 2^{-k(\gamma-\e)}$ for all $k\in \mathbb{N}$.
  We condition on $\{W_{0^j}: j\in \mathbb{N}\}$ and let  $A$ be this number.

  Given $0<r<1/2$, set $K = \lfloor \log_2(1/r)\rfloor$. Then, covering by intervals of lengths $1/r$,
  \begin{align*}
    \E\big(N_r(f(E^\alpha\cap [0,2^{-K}]))\big) &\leq \E \big(r^{-1} 2^{-K}W_0 W_{00}\ldots W_{0^{K}}L_{0^{K}}\big)\\
    &\leq r^{-1} 2^{-K}A\, 2^{-K(\gamma-\e)}\E (L_{0^{K}})\\
    &\leq A\,r^{-1} 2^{1+\gamma-\e} r^{1+\gamma -\e}\E (L)\\
    &= A\, 2^{1+\gamma-\e}\E (L)\, r^{\gamma -\e}.
  \end{align*}
  Thus, using Lemma \ref{thm:newprelim2}, taking $a$ as this random $A$ and the same $\e$,
  \begin{align*}
    \E\big(N_r(f(E^\alpha\cap [0,1]))\big)
    &\leq \E\big(N_r(f(E^\alpha\cap [0,2^{-K}]))\big) +\sum_{k=1}^K \E\big(N_r(f(E^\alpha\cap [2^{-k}, 2^{-k+1}]))\big)\\
    &\leq A\, 2^{1+\gamma-\e}\E (L)\, r^{\gamma -\e} + K c(t_1,t_2) r^{-\phi(t_1,t_2,\e)} +K^2\\
    &\leq A\, 2^{1+\gamma-\e}\E (L)\, r^{\gamma -\e} + \log_2(1/r)c(t_1,t_2) r^{-\phi(t_1,t_2,\e)} + ( \log_2(1/r))^2\\
    &= O\big(r^{-\phi(t_1,t_2,\e)}\log_2(1/r)\big)
  \end{align*}
  for small $r$.
  Hence, conditional on $\{W_0^j: j\in \mathbb{N}\}$, almost surely,
  $$\P\big(N_r(f(E^\alpha\cap [0,1]))\geq r^{-\phi(t_1,t_2,\e)-\delta} \big) 
  \leq r^{\delta /2}$$
  for $r$ sufficiently small, using Markov's inequality, so the Borel-Cantelli lemma taking $r=2^{-n}$ gives that 
  $N_r(f(E^\alpha\cap [0,1]))\leq r^{-\phi(t_1,t_2,\e)-\delta}$ for all sufficiently small $r$, almost surely.

  We conclude that, almost surely, for all $0<t_1<t_2<1$ with $t_2(\e)<t_2$,
  \begin{equation}\label{updim}
    \overline{\dim}_B(f(E_\alpha)) \leq 
    \sup_{x>0} \frac{1+\inf_{t\in [t_1,t_2]}\left(xt+\log_2\mathbb{E}(W^t)\right)}{1+x+(1+\gamma-\e)/\alpha} +\delta 
  \end{equation} 
  for all $\delta >0$. For $0<\tau< \min\{1/2,1-t_2(\e)\}$,
  $$
  \inf_{t\in [\tau,1-\tau]}(xt+\log_2\mathbb{E}(W^t))\leq \inf_{t\in [0,1]}(xt+\log_2\mathbb{E}(W^t)) + (x+ M)\tau,
  $$ 
  where $M$ is the maximum of the derivative of  $\mathbb{E}(W^t)$ over $[0,1]$.
  Substituting this in the numerator of \eqref{updim} with $t_1= \tau$ and $t_2= 1-\tau$, and noting that 
  $(x+M)/\big(1+x+(1+\gamma-\e)/\alpha\big)$ is bounded for $x>0$, we may let $\tau\searrow 0$, so
  that we may take the infima over $t\in [0,1]$ in \eqref{updim} and thus over $t>0$ using Lemma \ref{thm:monotone}(a).
  We may then let $\delta\searrow 0$ in \eqref{updim} and finally let $\e\searrow 0$, using the
  continuity in $\e$ from Lemma \ref{thm:monotone}(b), to get \eqref{updimallt}.
\end{proof}

\subsubsection{Lower  bound for $\dim_B f(E^\alpha)$}
To obtain the lower bound of Theorem \ref{thm:dimea} we establish a bound on the
distribution of the products $W_{i_1}\dots W_{i_1\dots i_{n}}$ of independent random variables on a binary tree.
We will use a well-known relationship between the free energy of the Mandelbrot
measure that goes back to Mandelbrot \cite{Mandelbrot74} and has been proved in a very general
setting in Attia and Barral \cite{Attia14}.

\begin{prop}[Attia and Barral \cite{Attia14}]\label{thm:attia}
  Let $X$ be a random variable with finite logarithmic moment function $\Lambda(q) = \log\E(e^{q X})$ for
  all $q\geq 0$. Write $R(x) = \inf_{q\in\R} (\Lambda(q) - xq)$ for the rate function and
  assume that $\Lambda(q)$ is twice differentiable for $q>0$.
  If  $\{X_\bi: \bi\in \cup_{j=1}^\infty \{0,1\}^j\}$ are independent and identically distributed
  with the distribution of $X$, then,
  \[
    \lim_{\eps\to0}\lim_{n\to\infty} \frac{1}{n}\log_2 \#\Big\{ \bi\in \{0,1\}^n :
    \sum_{j=1}^{n}X_{i_1\dots i_j}\in[n(x-\eps),n(x+\eps)]\Big\}
    =1+\frac{R(x)}{\log 2}.
  \]
\end{prop}
We refer the reader to the well-written account of the history of this statement in \cite{Attia14},
where Proposition \ref{thm:attia} is a special case of their Theorem 1.3(1), see in particular
(1.1) and situation (1) discussed in \cite[page 142]{Attia14}.
Note that the application of this theorem requires the strongest assumptions thus far on the random
variable $W$.

We derive a version of this Proposition suited to our setting.

\begin{lem}\label{thm:GW}
  Let $\eps,\delta>0$ and $0<q_0<1$, and choose $0<x<\gamma$ such that $\inf_{t>0}\left(
  2^{1+xt}\E(W^t) \right)>1$. Then there exists $n_0\in\mathbb{N}$ such that
  \begin{align}
    \mathbb{P}\bigg( \#\Big\{ \bi\in\{0,1\}^{n} : W_{i_1}\dots W_{i_1\dots i_{n}}&\geq 2^{-(x+\delta)n} \Big\}\nonumber \\
      &\geq 2^{-\eps n}\Big(\inf_{t>0}\big( 2^{1+xt}\E(W^t) \big)
    \Big)^{n}\;\text{for all}\;n\geq n_0 \bigg)  \geq q_0.\label{lbprob}
  \end{align}
\end{lem}

\begin{proof}
  Using Proposition \ref{thm:attia} with $X=\log_2 W$, $\Lambda(t) = \log \E(e^{t \log_2
  W})=\log_2\E(W^t)$, $R(x) = \inf_{t\in\R}\big(\log_2\E(W^t)-xt\big)$, and replacing $x$ by $-x$, we see that almost surely, 
  \begin{multline*}
    \lim_{\delta\to 0}\lim_{n\to\infty}\frac{1}{n}\log_2\#\left\{ \i\in\{0,1\}^{n} : W_{i_1}\dots W_{i_1 \dots i_{n}} \in \big[
    2^{-(x+\delta) n},2^{-(x-\delta)n}\big]\right\}
    \\=1+\inf_{t\in\R} \big(xt+\log_2\mathbb{E}(W^t)\big)
    =\log_2\inf_{t\in\R} 2^{1+xt}\mathbb{E}(W^t).
  \end{multline*}
  Since we are, for the moment, restricting to $0<x<\gamma$, we can
  assume that the infimum occurs when $t>0$ by Lemma \ref{thm:monotone}

  Since the event  $W_{i_1}\dots W_{i_1 \dots i_{n}} \in [2^{-(x+\delta) n},2^{-(x-\delta)n}]$
  decreases as $\delta\to 0$, for all  $\delta>0$, almost surely,
  \[  \lim_{n\to\infty}\frac{1}{n}\log_2\#\left\{ \i\in\{0,1\}^{n} : W_{i_1}\dots W_{i_1 \dots i_{n}} \in \big[
    2^{-(x+\delta) n},2^{-(x-\delta)n}\big]\right\}
    \geq\log_2\inf_{t\in\R} 2^{1+xt}\mathbb{E}(W^t).
  \]
  By Egorov's theorem, there exists $n_0$ such that with probability at least $q_0$, 
  \[\frac{1}{n}\log_2\#\left\{ \i\in\{0,1\}^{n} : W_{i_1}\dots W_{i_1 \dots i_{n}} \in \big[
    2^{-(x+\delta) n},2^{-(x-\delta)n}\big]\right\}
    \geq\log_2\inf_{t\in\R} 2^{1+xt}\mathbb{E}(W^t)-\epsilon.
  \]
  for all $n\geq n_0$, from which \eqref{lbprob} follows.
\end{proof}

We now develop Lemma \ref{thm:GW} to consider the independent subtrees with nodes a little way down
the main binary tree to get the probabilities to converge to 1 at a geometric rate. When we apply
the following lemma, we will take $\eps,\delta$ to be small and $\lambda$ close to 1.

\begin{lem}\label{thm:lb2}
  Assume that $\E(W^{-u})<\infty$ for some $u>0$.  Let $0<x<\gamma$ be such that $\inf_{t>0}\left(
  2^{1+xt}\E(W^t) \right)>1$, and let $\eps>0$ be sufficiently small so that $2^{-\eps} \inf_{t>0}(
  2^{1+xt}\E(W^t) \big)>1$. Let $\delta>0$ and $0<\lambda<1$. Then there exists $\eta>0$,
  $0<\theta<1$ and $k_0\in\mathbb{N}$, such that for all $k\geq k_0$, 
  \begin{align}
    \mathbb{P}\bigg( \#\Big\{ \bi\in\{0,1\}^{k} :  W_{i_1}\dots W_{i_1\dots i_{k}}&L_{i_1\dots i_{k}}\geq 2^{-(x+\delta)\lceil\lambda k\rceil -\eta \lfloor(1-\lambda)k \rfloor} \Big\}\nonumber \\
      &\geq (1-p/2) 2^{-\eps \lceil\lambda k\rceil}\Big(\inf_{t>0}\big( 2^{1+xt}\E(W^t) \big)
    \Big)^{\lceil\lambda k\rceil} \bigg)  \geq 1-\theta^k,\label{lbprob2}
  \end{align}
  where $p =\P(L \geq 1) >0$.
\end{lem}

\begin{proof}
  Fix some $0<q_0<1$ and let  $k \geq k_0$ where $\lceil\lambda k_0\rceil\geq n_0$, with $n_0$ given
  by  Lemma \ref{thm:GW}.  At level $\lfloor(1-\lambda)k \rfloor$ of the binary tree there are
  $2^{\lfloor(1-\lambda)k \rfloor}$ nodes of subtrees which have depth $\lceil\lambda k\rceil$. By
  Lemma \ref{thm:GW}, for each node $\bj \in \{0,1\}^{\lfloor(1-\lambda)k \rfloor}$, there is a
  probability of at least $q_0$ such that its subtree of depth $\lceil\lambda k\rceil$ has
  `sufficiently many paths with a large $W$ product', that is with
  \begin{align}
    \#\Big\{ \bi'\in\{0,1\}^{\lceil\lambda k\rceil} : W_{\bj i'_1}\dots W_{\bj i'_1\dots i'_{\lceil\lambda k\rceil}}&\geq 2^{-(x+\delta)\lceil\lambda k\rceil} \Big\}\label{subtreea}\\
    &  \geq 2^{-\eps \lceil\lambda k\rceil}\Big(\inf_{t>0}\big( 2^{1+xt}\E(W^t) \big)
    \Big)^{\lceil\lambda k\rceil}. \label{subtree} 
  \end{align}
  Since these subtrees are independent, the probability that none of them satisfy \eqref{subtree} is
  at most $(1-q_0)^{2^{\lfloor(1-\lambda)k \rfloor}}\leq \theta_0^k$ for some $0<\theta_0 <1$.
  Otherwise, at least one subtree satisfies \eqref{subtree}, say one with node $\bj$ for some $\bj
  \in \{0,1\}^{\lfloor(1-\lambda)k \rfloor}$, choosing the one with minimal binary string if there
  are more than one. We condition on this $\bj$ existing, which depends only on $\{W_{\bi} :
  \lfloor(1-\lambda)k \rfloor < |\bi| \leq k\}$.

  Choose $\eta>0$ such that  
  $2^{-\eta u} \E(W^{-u})<1$. Using Markov's inequality,
  \[ 
    \P\Big(W_{j_1} \ldots W_{\bj}<2^{-\eta \lfloor(1-\lambda)k \rfloor}\Big)
    < \big(2^{-\eta u} \E(W^{-u})\big)^{\lfloor(1-\lambda)k \rfloor}.
  \]
  Let $M\geq 2^{-\eps \lceil\lambda k\rceil}\big(\inf_{t>0}\big( 2^{1+xt}\E(W^t) \big)
  \big)^{\lceil\lambda k\rceil} >1$ be the (random) number in \eqref{subtreea}. Recalling that
  $\P(L_{\bi} \geq 1) = p>0$ for all $\bi$, and using a standard binomial distribution estimate
  coming from Hoeffding's inequality (see Lemma \ref{thm:hoeffding}), 
  \begin{align}
    \P\Big(\big\{\# \bi'\in\{0,1\}^{\lceil\lambda k\rceil}\text{ satisfying \eqref{subtreea} with }&
    L_{\bj\bi'} < 1\big\} \geq M(1-p/2)\Big) \leq \exp\big(-\textstyle{\frac{1}{2}} p^2 M\big)\nonumber\\
    &\leq  \exp\Big(-2^{-1}\big(2^{-\eps} \inf_{t>0}( 2^{1+xt}\E(W^t) \big)^{\lceil\lambda k\rceil}p^2 \big)\Big). \nonumber
  \end{align}

  Hence, conditional on $\bj$, 
  \begin{align}
    \#\Big\{ \bi'\in\{0,1\}^{\lceil\lambda k\rceil} : W_{j_1} \ldots W_{\bj}W_{\bj i'_1}\dots W_{\bj\bi'}&L_{\bj\bi'}\geq 2^{-(x+\delta)\lceil\lambda k\rceil-\eta \lfloor(1-\lambda)k \rfloor} \nonumber\Big\}\\
    & \geq (1-p/2) 2^{-\eps \lceil\lambda k\rceil}\Big(\inf_{t>0}\big( 2^{1+xt}\E(W^t) \big)
    \Big)^{\lceil\lambda k\rceil}\label{subtreez}
  \end{align}
  with probability at least 
  \[1- \big(2^{-\eta u} \E(W^{-u})\big)^{\lfloor(1-\lambda)k \rfloor}
    -\exp\Big(-2^{-1}\big(2^{-\eps} \inf_{t>0}( 2^{1+xt}\E(W^t) \big)^{\lceil\lambda
    k\rceil}p_L\big)\Big)
    \geq 1-c_1\theta_1^k,
  \]
  for some $0<\theta_1<1$ and $c_1>0$, for all $k\geq k_0$.

  The conclusion \eqref{lbprob2} now follows, since the unconditional probability of \eqref{subtreez} is at least
  $1- \theta_0^k - c_1\theta_1^k \geq 1-\theta^k$, on choosing $\max\{\theta_0,\theta_1\}<\theta<1$,
  and increasing $k_0$ if necessary to ensure that 
  $\theta^k \geq \theta_0^k +c_1\theta_1^k$ for all $k\geq k_0$.
\end{proof}
Using Lemma \ref{thm:lb2} we can obtain the lower bound for Theorem \ref{thm:dimea}.

\begin{prop}\label{proplb}
  Let $\alpha>0$. Under the assumptions in Theorem \ref{thm:dimea}, almost surely,
  \[
    \lbd f(E^\alpha) \geq \sup_{x>0} \frac{1+\inf_{t>0}\left(x
    t+\log_2\mathbb{E}(W^t)\right)}{1+x+(1+\gamma)/\alpha}.
  \]
\end{prop}
\begin{proof}
  Fix $0<x<\gamma$ and let $\eps,\delta, \eta, \lambda, \theta$ be as in Lemma \ref{thm:lb2}.
  For $k\in \mathbb{N}$ let $\bl_k := 0^{k-1}1 \in \{0,1\}^k$.
  Replacing  $k$ by $\lfloor \alpha k \rfloor$ in  \eqref{lbprob2} and noting that $\sum_1^\infty
  \theta^{\lfloor \alpha k \rfloor} <\infty$, it follows from the Borel-Cantelli lemma that almost
  surely there exists a random $K_1<\infty$ such that for all $k\geq K_1$,
  \begin{align}
    \#\Big\{ \bi\in\{0,1\}^{\lfloor \alpha k\rfloor} :  W_{\bl_k i_1}\dots W_{\bl_k \bi}L_{\bl_k \bi}&\geq a2^{-(\lambda(x+\delta)+\eta (1-\lambda))\lfloor \alpha k\rfloor} \Big\}\nonumber\\
    & \geq b 2^{-\eps\lambda \lfloor \alpha k\rfloor}\Big(\inf_{t>0}\big( 2^{1+xt}\E(W^t) \big)
    \Big)^{\lambda \lfloor \alpha k\rfloor}.\label{lbprob2x}
  \end{align}
  Here the numbers $a,b>0$, which are introduced for notational convenience so we can replace
  $\lceil\lambda k\rceil$ by $\lambda k$ and $\lfloor(1-\lambda)k \rfloor$ by $(1-\lambda)k$ in
  \eqref{lbprob2}, depend on $x, \eps,\delta, \eta, \lambda$ but not on $k$. 

  By the strong law of large numbers, $(W_{0}W_{00}\dots W_{0^{k-1}}W_{\bl_k})^{1/k} \to 2^{-\gamma}$ almost
  surely, so almost surely there exists $K_2\in\mathbb{N}$ such that $W_{0}W_{00}\dots
  W_{0^{k-1}}W_{\bl_k}
  \geq 2^{-(\gamma-\eps)k}$  for all $k\geq K_2$.

  For $k \in\mathbb{N}$ let
  \[
    r_k=2^{-(k+\lfloor \alpha k\rfloor)} \cdot 2^{-(\gamma-\eps)k} \cdot
    a2^{-(\lambda(x+\delta)+\eta (1-\lambda))\lfloor \alpha k\rfloor}.
  \]
  Then
  \begin{align*}
    N_{r_k}(f(E^\alpha)) 
    &\geq \#\big\{ \bj=\bl_k\bi 0\dots
      \in\Sigma_\alpha : \bi\in\{0,1\}^{\lfloor \alpha k\rfloor} ,
    |f(I_{\bj})|\geq r_k\big\}
    \\&\geq
    \#\big\{ \bj= \bl_k\bi  : \bi\in\{0,1\}^{\lfloor \alpha k\rfloor}  ,
    2^{-(k+\lfloor \alpha k\rfloor)}W_{j_1}W_{j_1 j_2}\dots W_{\bj}L_{\bj}\geq r_k\big\}
    \\&
    \geq 
    \#\big\{\bi\in\{0,1\}^{\lfloor \alpha k\rfloor}  : W_{0}W_{00}\dots W_{0^{k-1}}W_{\bl_k}\geq 2^{-(\gamma-\eps)k}\\
      &
      \qquad\qquad \text{ and } W_{\bl_k i_1}\dots W_{\bl_k \bi}L_{\bl_k \bi}\geq
    a2^{-(\lambda(x+\delta)+\eta (1-\lambda))\lfloor \alpha k\rfloor}\big\}\\
    & 
    \geq b 2^{-\eps\lambda \lfloor \alpha k\rfloor}\Big(\inf_{t>0}\big( 2^{1+xt}\E(W^t) \big)
    \Big)^{\lambda \lfloor \alpha k\rfloor},
  \end{align*}
  provided that $k\geq \max\{K_1,K_2\}$, using \eqref{lbprob2x}.

  Since $r_k\searrow 0$ no faster than geometrically, it suffices to compute the (lower) box-counting
  dimension along the sequence $r_k$. Hence
  \begin{align*}
    \lbd f(E^\alpha)
    &\geq
    \liminf_{k\to\infty} \frac{\log_2 N_{r_k}(f(\pi\Sigma_\alpha))}{-\log_2 r_k}
    \\
    &\geq  \liminf_{k\to\infty}
    \frac{\log_2 b -\eps\lambda \lfloor \alpha k\rfloor +  \lambda \lfloor \alpha k\rfloor  \log_2  \inf_{t>0}\big( 2^{1+xt}\E(W^t) \big)}
    {(k+\lfloor \alpha k\rfloor) + (\gamma-\eps)k  + (\lambda(x+\delta)+\eta (1-\lambda))\lfloor \alpha k\rfloor -\log_2 a }
    \\
    & = \frac{-\eps \lambda \alpha +\lambda \alpha\big(1+ \inf_{t>0}(xt+\log_2\E(W^t))\big)}
    {1+ \alpha  +\gamma-\eps +  \alpha(\lambda(x+\delta)+\eta (1-\lambda))}
    \\
    &= \frac{\lambda(1-\eps)  +\lambda \big(\inf_{t>0}(xt+\log_2\E(W^t))\big)}
    {1+ (1 +\gamma-\eps)/\alpha +  \lambda(x+\delta)+\eta (1-\lambda)}
  \end{align*}
  almost surely, on letting $k\to \infty$ and dividing through by $\alpha$. This is valid for all
  $\eps,\delta>0$ and $0<\lambda<1$, so  we obtain
  \begin{equation}\label{eq:almostThere}
    \lbd f(E^\alpha)
    \geq \frac{1+\inf_{t>0}(x t +\log_2 \E(W^t))}{1+x+(\gamma+1)/\alpha}
  \end{equation}
  for all $0<x<\gamma$.
  However, for $x\geq \gamma$
  the infimum in \eqref{eq:almostThere} is $0$ by Lemma \ref{thm:monotone}, whereas the denominator is increasing in $x$. Thus
  the supremum is achieved taking $0<x<\gamma$, as required.
\end{proof} 

\begin{proof}[Proof of Theorem \ref{thm:dimea}]
  For fixed $\alpha$, Theorem \ref{thm:dimea} follows immediately from Propositions \ref{propnew}  and \ref{proplb}.
  Further, with probability $1$, \eqref{dimgen} holds simultaneously for all
  countable subsets $A\subset (0,\infty)$ and so in particular for $\mathbb{Q}^+$. Since
  \eqref{dimgen} is continuous in $p$, it must hold for all $p>0$ simultaneously and so Theorem \ref{thm:dimea}
  holds.
\end{proof}

\subsubsection{Box dimension of  $f(E_\ba)$ for  $\ba \in S_p$ }

It remains to extend  Theorem \ref{thm:dimea} to Theorem \ref{thm:dimgen} which we do using the `eventually separating' notion.

\begin{proof}[Proof of Theorem \ref{thm:dimgen}]

  For $\alpha >0$ let
  $$\phi(\alpha)= \sup_{x>0} \frac{1+\inf_{t>0}\left(xt+\log_2\mathbb{E}(W^t)\right)}{1+x+(1+\gamma)/\alpha}.$$
  Let $\ba \in S_p$ for $p>0$ and let $0< p_1 <p<p_2$. Then $E^{1/p_1} \in S_{p_1}$ and   $E^{1/p_2}
  \in S_{p_2}$, see \eqref{aln}. By Lemma \ref{thm:decreasingseqs},  $E^{1/p_1}$ eventually separates $\ba$, and 
  $\ba$ eventually separates $E^{1/p_2}$. Since $f$ is almost surely monotonic, it preserves
  `eventual separation' for all pairs of sequences, so $f(E^{1/p_1})$ eventually separates $f(\ba)$  and $f(\ba)$ eventually
  separates $f(E^{1/p_2})$. By Lemma \ref{thm:eventuallySeparates},
  $$\phi(1/p_2)\leq \bdd  f(E^{1/p_2})  \leq  \lbd f(E_\ba) \leq \ubd f(E_\ba)\leq \bdd f(E^{1/p_1})\leq \phi(1/p_1).$$
  By Lemma \ref{thm:monotone} $\phi$ is continuous in $\alpha$, so taking $p_1,p_2$ arbitrarily
  close to $p$, we conclude that $\bdd f(E_\ba) = \phi(1/p)$.

  Further, since `eventual  separation' is preserved almost surely for all pairs of sequences $\ba$
  and $\ba'$, the box-counting dimension of $E_{\ba}$ is constant for all $\ba\in S_p$.
  Applying Theorem \ref{thm:dimea} get that $\bdd f(E_{\ba}) = \phi(1/p)$ for all $\ba\in S_p$ and
  $p>0$ simultaneously with probability $1$.
\end{proof}

\subsection{Decreasing sequences}
\label{sec:otherproofs}
We now prove the statements in Section \ref{sec:decreasingsets}.

\begin{proof}[Proof of Lemma \ref{thm:eventuallySeparates}]
  We may assume that $n_0=1$ in the definition of $\bb$ eventually separating $\ba$, since removing a finite number of points from a sequence does not affect
  its box-counting dimensions. For $r>0$ and $E$ a bounded subset of $\mathbb{R}$ let $N_r(E)$ be the maximal
  number of points in an $r$-separated subset of $E$, and let $\{a_{n_i}\}_{i=1}^{N_r(A)}$
  be a maximal $r$-separated subset of ${\ba}$ (with $n_i$ increasing). Then for each
  $1\leq i \leq N_r(A)-1$ there exists $b_{m_i} \in {\bb}$ such that  $a_{n_{i +1}}\leq b_{m_i}  \leq a_{n_i}$. Then 
  $\{b_{m_1},b_{m_3},b_{m_5},\ldots,b_{m_N}\}$ is an $r$-separated set, where $N$ is the largest odd
  number less than $N_r({\ba})$. It follows that
  $N_r({\bb}) \geq \frac{1}{2}(N+1)\geq \frac{1}{2}(N_r({\ba})-2)$. The inequalities
  now follow from the definition of the lower box-counting dimension $\lbd E = \varliminf_{r\to 0}
  \log N_r(E) /-\log r$, and similarly for upper box-counting dimension.
\end{proof} 

\begin{proof}[Proof of Theorem \ref{thm:decreasingseqs}]
  Given $\e>0$ there is $n_0\in \mathbb{N}$ such that if $n\geq n_0$ then 
  $$ n^{-p-\e} \ \leq a_{n+1} \ \leq a_n\  \leq n^{-p+\e}.$$ 
  Since that gaps of $\ba$ are decreasing, by comparing $a_n-a_{n+1}$ with the $n-\lfloor n^{1-\e} \rfloor$ gaps between $a_n$ and  $a_{\lfloor n^{1-\e} \rfloor}$, we see that
  $$a_n -a_{n+1} 
  \ \leq\  \frac{ a_{\lfloor n^{1-\e} \rfloor}-a_n}{n-\lfloor n^{1-\e} \rfloor}
  \ \leq\  \frac{ \lfloor n^{1-\e}\rfloor^{(-p+\e)}}{n-\lfloor n^{1-\e} \rfloor}
  \ \leq\ 2 n^{-p-1 +\e+\e^2}
  \ \leq\ 2 x^{(p+1 -\e-\e^2)/(p+\e)},
  $$
  for all $x\in [a_{n+1},a_n]$, for all sufficiently large $n$, equivalently all sufficiently small $x>0$.
  Hence by redefining $\e$,  given $\e>0$ the right-hand inequality of 
  \be\label{gapsa}
  x^{1+1/p +\e} \ \leq\ a_n -a_{n+1} \ \leq\  x^{1+1/p -\e}\qquad (x\in [a_{n+1},a_n])
  \ee
  holds for all sufficiently large $n$; the left-hand inequality following from a similar estimate. For the sequence $\bb$
  \[
    x^{1+1/q +\e} \ \leq\ b_m -b_{m+1} \ \leq\  x^{1+1/q -\e}\qquad (x\in [b_{m+1},b_m]).
  \]
  Choose $0<\e <\frac{1}{2}(\frac{1}{q} - \frac{1}{p})$, and take $x$ small enough, that is $n,m$ large enough, for \eqref{gapsa} and  \eqref{gapsa} to hold. For such an $n$, choose $x\in [a_{n+1},a_n]$. Taking $m$ such that $x\in [b_{m+1},b_m]$,
  $$b_m -b_{m+1} \ \leq\  x^{1+1/q -\e}\ <\  x^{1+1/p +\e} \ \leq\ a_n -a_{n+1}.$$
  Thus the interval $[a_{n+1},a_n]$ intersects the shorter interval $[b_{m+1},b_m]$, so either $b_{m} \in [a_{n+1},a_n]$ or $b_{m+1} \in [a_{n+1},a_n]$, so $\bb$ eventually separates $\ba$.
\end{proof}

\subsection*{Acknowledgements}
The authors thank the anonymous referee for their many helpful suggestions that improved this
manuscript.
The authors further thank Xiong Jin for his comments on an earlier draft.

\bibliographystyle{plain}

\end{document}